\documentclass[11pt]{article}

\usepackage{amsmath,amsthm,amssymb,amsfonts, graphicx}
\usepackage{graphics}
\usepackage{authblk}
\usepackage{upgreek}
\usepackage{amsrefs,hyperref}
\theoremstyle{plain}
\newtheorem{theorem}{Theorem}[section]

\newtheorem{corollary}[theorem]{Corollary}
\newtheorem{definition}[theorem]{Definition}

\newtheorem{lemma}[theorem]{Lemma}
\newtheorem{proposition}[theorem]{Proposition}
\newtheorem{conjecture}[theorem]{Conjecture}
\theoremstyle{definition}
\newtheorem{remark}[theorem]{Remark}
\numberwithin{equation}{section}
\newcommand{\diff}{\mathop{}\!\mathrm{d}}
\DeclareMathOperator{\Hess}{Hess}
\DeclareMathOperator{\tr}{tr}
\DeclareMathOperator{\divr}{div}
\DeclareMathOperator{\spn}{span}
\DeclareMathOperator{\supp}{supp}

\DeclareMathOperator{\sgn}{sgn}

\title{A Helmholtz-type decomposition for the space of symmetric matrices}
\author[1]{Evan Miller}
\author[2]{Eric Sawyer}
\affil[1]{University of British Columbia, Department of Mathematics

emiller@msri.org}

\affil[2]{McMaster University,
Department of Mathematics and Statistics

sawyer@mcmaster.ca}

\begin{document}

\maketitle

\begin{abstract}
In this paper, we introduce a Helmholtz-type decomposition for the space of square integrable, symmetric-matrix-valued functions analogous to the standard Helmholtz decomposition for vector fields. This decomposition provides a better understanding of the strain constraint space, which is important to the Navier--Stokes regularity problem. In particular, we give a full characterization the orthogonal complement of the strain constraint space and investigate the geometry of the eigenvalue distribution of matrices in the strain constraint space.
\end{abstract}

\section{Introduction}

The Navier--Stokes equation is one of the most important equations in fluid dynamics. There is a copious literature on the role of the vorticity---which represents the anti-symmetric part of the gradient of velocity and describes rotation by the motion of the fluid---in the evolution of solutions of the Navier--Stokes equation. There has been much less work focused on the role of the strain matrix, which is the symmetric part of the gradient of velocity, and describes deformations by the motion of the fluid.

In general, most of the focus on strain has been on its role in vortex stretching and its nonlocal dependence on vorticity, neglecting the role of strain self-amplification, which increasingly appears to play an important role in its own right.
Recently, the first author considered mild solutions of the strain evolution equation \cite{MillerStrain}, providing an alternative proof for a geometric regularity criterion in terms of the positive part of the middle eigenvalue of the strain first proven by Neustupa and Penel in \cite{NeustupaPenel}. The first author later proved blowup for a toy model for the strain that respects the constraint space and the identity for enstrophy growth, giving a perturbative condition for blowup for the full Navier--Stokes equation \cite{MillerStrainToyModel}. 
In addition to these works studying the role of the strain matrix utilizing methods from the analysis of PDEs, there has also been work on this topic in the fluid mechanics literature. Carbone and Bragg recently showed that the self-amplification of strain was a more important feature of the turbulent energy cascade than vortex stretching \cite{CarboneBragg}, which had been previously assumed to play the dominant role, both in the turbulent energy cascade and in the Navier--Stokes regularity problem. 

In light of these recent works suggesting that the strain plays an essential role in the Navier--Stokes regularity problem, as well as the phenomenology of turbulence, it is necessary to develop a more complete and detailed understanding of the strain constraint space. In \cite{MillerStrain}, the first author defined the space of strain matrices as the symmetric gradients of divergence free vector fields
\begin{equation}
    \mathcal{L}^2_{st}
    =
    \nabla_{sym}\dot{\mathbf{H}}^1_{df},
\end{equation}
and gave the following characterisation of this constraint space.

\begin{theorem} \label{CharacterizationARMA}
    For all $S\in\mathcal{L}^2$,
    $S\in \mathcal{L}^2_{st}$ if and only if
    \begin{equation}
        \tr(S)=0
    \end{equation}
    and
    \begin{equation}
        S+2\nabla_{sym}\divr(-\Delta)^{-1} S=0.
    \end{equation}
\end{theorem}

While Theorem \ref{CharacterizationARMA} gives a complete characterization of the strain constraint space $\mathcal{L}^2_{st}$, it does not provide a decomposition of the space of symmetric matrices equivalent to the Helmholtz decomposition. The Helmholtz decomposition, which was also developed for applications in fluid mechanics as well as electromagnetism, states that any vector field can be written as the sum of a gradient and a divergence free vector field. In particular, Helmholtz proved in \cite{Helmholtz} that 
if $u\in C^1\left(\mathbb{R}^3;\mathbb{R}^3\right)$,
then $u$ is the sum of a gradient and a curl, with
\begin{equation}
    u=-\nabla \phi +\nabla \times v.
\end{equation}
This was first proven in 1858 when the notion of infinite dimensional Hilbert spaces and in particular the Lebesgue integral \cite{Lebesgue} did not yet exist, as both would revolutionize analysis 
in the early 20\textsuperscript{th} century.
The modern statement of the Helmholtz decomposition is given below.

\begin{theorem} \label{HelmholtzOld}
The space of square integrable vector fields 
$L^2\left(\mathbb{R}^3;\mathbb{R}^3\right)$ has the following orthogonal decomposition.
\begin{equation}
    \mathbf{L}^2=
    \mathbf{L}^2_{df} \oplus \mathbf{L}^2_{gr}.
\end{equation}
\end{theorem}

The Helmholtz decomposition for vector fields tells us that any vector field can be written as the sum of a gradient and a divergence free vector field.
We will prove an analogous result that any symmetric matrix valued function can be written as the sum of a Hessian, a strain matrix, and a divergence free matrix.

\begin{theorem} \label{HelmholtzOneIntro}
The space of square integrable symmetric matrices
$L^2\left(\mathbb{R}^3;\mathcal{S}^{3\times 3}\right)$
has the following orthogonal decomposition.
\begin{equation}
\mathcal{L}^2=
\mathcal{L}^2_{st} \oplus \mathcal{L}^2_{Hess} 
\oplus \mathcal{L}^2_{divfree}.
\end{equation}
\end{theorem}

This orthogonal decomposition is one of the central results of this paper. Before we can discuss this result further it is necessary to define the divergence operator for vectors and matrices, and also to define a great number of spaces. For a vector-valued function $v\in \mathbb{R}^d$, the divergence is a scalar, given by 
\begin{equation}
    \nabla \cdot v= \sum_{i=1}^d \partial_i v_i.
\end{equation}
For a matrix-valued function $M\in \mathbb{R}^{d\times d}$, the divergence is a vector-valued function $\divr(M)\in \mathbb{R}^d$ with components
\begin{equation}
    \divr(M)_j=\sum_{i=1}^d \partial_i M_{ij}.
\end{equation}
We will generally use the notation $\divr$ to refer to the divergence of a matrix, and the notation $\nabla \cdot$ to refer to the divergence of a vector field. We will refer to the divergence squared of a matrix as
\begin{equation}
    \divr^2(M)=\nabla \cdot \divr(M)
    =\sum_{i,j=1}^d \partial_i \partial_j M_{ij}
\end{equation}
For matrices $M,Q\in \mathbb{R}^{3\times 3}$,
we will define the matrix inner product by
\begin{equation}
    M \odot Q
    =
    \sum_{i,j=1}^3
    M_{ij}Q_{ij}.
\end{equation}

Throughout this paper, we will use the convention that $L^{2}$, $\mathbf{L}^{2}$ and $\mathcal{L}^{2}$ are the space of square integrable scalars, vectors and symmetric matrices respectively on $\mathbb{R}^{3}$
with respect to Lebesgue measure.
We will also use the convention that 
\begin{equation}
  \mathcal{S}^{3\times 3}=\left\{S\in \mathbb{R}^{3\times 3}: S_{ij}=S_{ji}, 1\leq i,j \leq 3\right\}  
\end{equation} 
will refer to the space of symmetric, real-valued matrices.
The main spaces of symmetric matrices arising in our study of the structure of the strain constraint space are strain matrices,
Hessians, scalar multiples of the identity matrix, trace free matrices, divergence free matrices, and divergence squared free matrices. We will first define the space of divergence free vector fields and the space of gradients, and then proceed to define each of these spaces of matrices.

\begin{definition} \label{BigDef}
We will now define the spaces of matrices listed above, as well as an ``adjusted'' identity space, and the space of gradients and divergence free vector fields.

    \begin{equation}
    L^2=L^2\left(\mathbb{R}^3;\mathbb{R}\right)    
    \end{equation}
    
    \begin{equation}
        \mathbf{L}^2=L^2\left(\mathbb{R}^3;\mathbb{R}^3\right).
    \end{equation}
    
    \begin{equation}
        \mathbf{L}^2_{df}=\left\{u\in \mathbf{L}^2: \nabla \cdot (-\Delta)^{-\frac{1}{2}}u=0\right\}.
    \end{equation}
    
    \begin{equation}
        \mathbf{L}^2_{gr}=\left\{\nabla (-\Delta)^{-\frac{1}{2}}f: f\in L^2\right\}.
    \end{equation}
    
\begin{equation}
    \mathcal{L}^2= L^2\left(\mathbb{R}^3;
    \mathcal{S}^{3\times 3}\right).
\end{equation}
    
\begin{equation}
    \mathcal{L}^2_{st}=\left\{\nabla_{sym}
    (-\Delta)^{-\frac{1}{2}}u: u\in\mathbf{L}^2_{df}\right\}.
\end{equation}

\begin{equation}
    \mathcal{L}^2_{Hess}=\left\{\Hess(-\Delta)^{-1}f:
    f\in L^2\right\}.
\end{equation}

\begin{equation}
    \mathcal{L}^2_{Id}=\left\{f I_3:
    f\in L^2\right\}.
\end{equation}

\begin{equation}
    \mathcal{L}^2_{\widetilde{Id}}=\left\{fI_3
    +\Hess(-\Delta)^{-1}f:
    f\in L^2\right\}.
\end{equation}

\begin{equation}
     \mathcal{L}^2_{divfree} = \ker(\divr)=
     \left\{M\in\mathcal{L}^2: 
     \divr(-\Delta)^{-\frac{1}{2}}M=0\right\}.
\end{equation}

\begin{equation}
     \mathcal{L}^2_{div^2free} = \ker\left(\divr^{2}\right)
     = \left\{M\in\mathcal{L}^2: 
     \divr^2(-\Delta)^{-1}M=0\right\}.
\end{equation}

\begin{equation}
    \mathcal{L}^2_{trfree}=\ker(\tr)=
    \left\{M\in\mathcal{L}^2: 
     \tr(M)=0 \right\}.
\end{equation}

\begin{equation}
    \mathcal{L}^2_{tr\&divfree}=\ker(\tr) \cap \ker(\divr)
    =\left\{M\in\mathcal{L}^2: \tr(M)=0, 
    \divr (-\Delta)^{-\frac{1}{2}}M=0\right\}.
\end{equation}

We will also define homogeneous Sobolev norm,
for all $s>-\frac{3}{2}$, by
\begin{equation}
    \|f\|_{\dot{H}^s}
    =
    \left(\int_{\mathbb{R}^3}
    \left(4\pi^2|\xi|^2\right)^s
    |\hat{f}(\xi)|^2 \diff\xi
    \right)^\frac{1}{2},
\end{equation}
and we will define the homogeneous Sobolev space $\dot{H}^s\left(\mathbb{R}^3\right)$ as the closure of $C_c^\infty\left(\mathbb{R}^3\right)$ with respect to this norm.
Furthermore, we will define the subspaces $\dot{\mathbf{H}}^s_{df}, \dot{\mathbf{H}}^s_{gr}, \dot{\mathcal{H}}^s_{st},$ etc. analogously to the corresponding $L^2$ definitions above.
\end{definition}

We introduce the space $\mathcal{L}^2_{\widetilde{Id}}$
in Definition \ref{BigDef}, 
because we are trying to form an orthogonal decomposition for $\mathcal{L}^2$, and $\mathcal{L}^2_{Hess}$ and $\mathcal{L}^2_{Id}$ are both natural subspaces to consider, however they are not orthogonal to each other. The space $\mathcal{L}^2_{\widetilde{Id}}$ is the space of $L^2$ scalar multiples of the identity matrix projected to be orthogonal to $\mathcal{L}^2_{Hess}$,
and in fact we will show that
\begin{equation}
        \mathcal{L}^2_{\widetilde{Id}}=P_{Hess}^\perp 
        \mathcal{L}^2_{Id}.
\end{equation}

In previous work on the Navier--Stokes strain evolution equation \cite{MillerStrain}, the first author showed that identities and Hessians are both orthogonal to the space of strain matrices, with
\begin{equation}
    \mathcal{L}^2_{Hess}, \mathcal{L}^2_{Id}
    \subset \left(\mathcal{L}^2_{st}\right)^\perp.
\end{equation}
Using the adjusted space $\mathcal{L}^2_{\widetilde{Id}}$ this can be restated as
\begin{equation}
    \mathcal{L}^2_{Hess} \oplus
    \mathcal{L}^2_{\widetilde{Id}}
    \subset
    \left(\mathcal{L}^2_{st}\right)^\perp.
\end{equation}
One natural question to ask, then, is whether or not the direct sum of the Hessians and adjusted identities gives us the entire orthogonal complement of the strain space. We will show that it does not in $\mathbb{R}^3$, although it turns out this direct sum does give the entire complement when working in $\mathbb{R}^2$, which we will discuss later. We see this by further breaking down the orthogonal decomposition.

\begin{theorem} \label{HelmholtzTwoIntro}
We have the further orthogonal decomposition
\begin{equation}
    \mathcal{L}^2_{divfree}
    =
    \mathcal{L}^2_{\widetilde{Id}} \oplus 
    \mathcal{L}^2_{tr\&divfree},
\end{equation}
and therefore the entire space $\mathcal{L}^2$ has the orthogonal decomposition
\begin{equation}
    \mathcal{L}^2=
    \mathcal{L}^2_{st} \oplus \mathcal{L}^2_{Hess} 
    \oplus L^2_{\widetilde{Id}} \oplus \mathcal{L}^2_{tr\&divfree}.
\end{equation}
\end{theorem}

As an immediate corollary of the orthogonal decompositions provided by Theorem \ref{HelmholtzTwoIntro}, we can describe the orthogonal complement of the strain space as follows.

\begin{corollary} \label{OrthogonalComplementIntro}
The orthogonal complement of the strain constraint space is given by
\begin{align}
    \left(\mathcal{L}^2_{st}\right)^\perp 
    &=
    \mathcal{L}^2_{Hess} \oplus
    \mathcal{L}^2_{\widetilde{Id}} \oplus
    \mathcal{L}^2_{tr\&divfree} \\
    &=
    \mathcal{L}^2_{Hess} 
\oplus \mathcal{L}^2_{divfree}.
\end{align}
\end{corollary}

This corollary will also allow us to provide a distributional definition of $L^2_{st},$ which we had lacked up until this point.

\begin{corollary} \label{Distribution}
Suppose $S\in \mathcal{L}^2$. Then $S\in \mathcal{L}^2_{st}$
if and only if, 
for all $f,g\in C_c^\infty(\mathbb{R}^3)$
\begin{align}
    \left<S,\Hess(f)\right> &=0 \\
    \left<S,g I_3\right> &=0,
\end{align}
and for all $M\in C_c^\infty
\left(\mathbb{R}^3;\mathcal{S}^{3\times 3}\right)$, with
$\tr(M)=0$ and $\divr(M)=0$,
\begin{equation}
    \left<S,M\right>=0.
\end{equation}
\end{corollary}

Consistent with Corollary \ref{Distribution},
we will define the strain space of distributions as follows.

\begin{definition} \label{DistributionDef}
Suppose $S\in \mathcal{D}'
\left(\mathbb{R}^3;\mathcal{S}^{3\times 3}\right)$, where
$\mathcal{D}'\left(\mathbb{R}^3;\mathcal{S}^{3\times 3}\right)$ is the dual space of $C_c^\infty
\left(\mathbb{R}^3;\mathcal{S}^{3\times 3}\right)$.
Then $S\in D'_{st}$ if and only if
for all $f,g\in C_c^\infty(\mathbb{R}^3)$
\begin{align}
    \left<S,\Hess(f)\right> &=0 \\
    \left<S,g I_3\right> &=0,
\end{align}
and for all $M\in C_c^\infty
\left(\mathbb{R}^3;\mathcal{S}^{3\times 3}\right)$, with
$\tr(M)=0$ and $\divr(M)=0$,
\begin{equation}
    \left<S,M\right>=0.
\end{equation}
\end{definition}

We also have another quite nice representation of the space of strain matrices---which follows from our Helmholtz-type decomposition of the space of symmetric matrices in Theorem \ref{HelmholtzOneIntro}---as the direct difference of the kernel of divergence squared and the kernel of divergence. 
In this sense, the strain matrices can be seen as the matrices that are divergence squared free, but do not have any divergence free part.

\begin{theorem} \label{DifferenceThm}
\begin{equation}
\mathcal{L}^2_{st}=\ker\left(\divr^2\right) \ominus \ker(\divr),
\end{equation}
and equivalently
\begin{equation}
    \ker\left(\divr^2\right)= 
    \ker(\divr) \oplus \mathcal{L}^2_{st}.
\end{equation}
\end{theorem}

In this vein, we also have a formula for the size of the projection onto $\mathcal{L}^2_{st}$.
For all $M\in\mathcal{L}^2,$
\begin{equation} \label{CurlFormula}
    \left\|P_{st}(M)\right\|_{L^2}^2=
    2\left\|\nabla\times\divr(-\Delta)^{-1}
    M\right\|_{L^2}^2.
\end{equation}
Crucial to this identity, will be an isometry proven by the first author in \cite{MillerStrain}, which is stated below.
\begin{proposition} \label{isometry}
For all $u\in \mathbf{L}^2_{df},$
\begin{equation}
    \left\|\nabla_{sym}(-\Delta)^{-\frac{1}{2}}
    u\right\|_{L^2}^2
    =
    \frac{1}{2}\|u\|_{L^2}^2
    =
    \frac{1}{2}\left\|\nabla\times(-\Delta)^{-\frac{1}{2}}
    u\right\|_{L^2}^2.
\end{equation}
\end{proposition}

Finally, we will also note that the space $\mathcal{L}^2_{st}$ has a natural rotation invariance.
In Chapter 1 of \cite{MajdaBertozzi}, Majda and Bertozzi remark that the space of divergence free vector fields has a rotational invariance, and apply this rotational invariance to the Euler and Navier--Stokes equations, stating the following.

\begin{proposition}
For all $u\in \mathbf{\dot{H}}^1_{df},$ and for all $Q\in SO(3)$, $u^Q \in \mathbf{\dot{H}}^1_{df},$ where
\begin{equation}
    u^Q(x,t)=Q^{tr}u\left(Qx,t \right).
\end{equation}
\end{proposition}

For strain matrices we have an analogous rotational invariance, although in this case the form of the rotation invariance is even more canonical, because the invariance involves rotating the domain and conjugating by the same rotation in the range. This has a nicer algebraic structure than the corresponding vector result, because of the additional algebraic structure of symmetric matrices.

\begin{proposition} \label{RotInvarIntro}
For all $S\in\mathcal{L}^2_{st}$, and for all
rotation matrices $Q\in SO(3),$ $S^Q \in\mathcal{L}^2_{st}$,
where
\begin{equation}
    S^Q(x)= Q^{tr} S(Qx) Q.
\end{equation}
\end{proposition}

This essentially sums up all of the major results that we will prove for the Helmholtz decomposition of symmetric matrices.
We will now apply these results to a problem that has a direct bearing on the Navier--Stokes regularity problem: what type of geometric structures with regards to the eigenvalue distribution are possible within $\mathcal{L}^2_{st}$, the strain constraint space?
In order for a solution of the Navier--Stokes equation to blowup in finite-time, the positive part of the middle eigenvalue must blowup in a number of scale-critical spaces.
In particular, if $T_{max}<+\infty$, then for all 
$\frac{2}{p}+\frac{3}{q}=2,
\frac{3}{2}<q \leq +\infty$,
\begin{equation} \label{RegCrit}
    \int_0^{T_{max}}\|\lambda_2^+(\cdot,t)\|_{L^q}^p 
    \diff t =+\infty,
\end{equation}
where $\lambda_1 \leq \lambda_2\leq \lambda_3$ are the eigenvalues of $S$,
and $\lambda_2^+=\max\left(0,\lambda_2\right)$.
This was first proven by Neustupa and Penel in \cite{NeustupaPenel} and later revisited by the first author in a paper considering the evolution equation for the strain \cite{MillerStrain}.
This regularity criterion relies on the following identity for enstrophy growth,
\begin{equation} \label{EntGrowth}
    \frac{\diff}{\diff t}\|S(\cdot,t)\|_{L^2}^2
    =-2\|S\|_{\dot{H}^1}^2-4\int\det(S).
\end{equation}
Recall that for all $S\in\mathcal{L}^2_{st}, \tr(S)=0$---this is the divergence free constraint---and so we can see that the term $-\int\det(S)$ will drive growth when there are two positive eigenvalues and one very negative eigenvalue, the case of planar stretching/axial compression. The regularity criterion $\eqref{RegCrit}$ suggests that for the fastest possible growth, we should want $\lambda_2$ as large as possible, and in particular should want $\lambda_2=\lambda_3$ 
(and therefore $\lambda_1=-2\lambda_3$).
Furthermore, the first author showed in \cite{MillerStrain} that for all 
$M\in \mathcal{S}^{3\times 3}, \tr(M)=0$,
\begin{equation}
    -4\det(M)\leq \frac{2}{9}\sqrt{6}|M|^3,
\end{equation}
with equality if and only if $\lambda_2=\lambda_3$.
It is extremely clear, therefore, that the fastest possible growth of enstrophy would occur when $\lambda_2=\lambda_3$ at all points in space, and that such strain matrices, if they in fact exist, would be very natural candidates for finite-time blowup for Navier--Stokes.

It is far from clear however, whether this eigenvalue structure can be attained, or nearly attained, within the strain constraint space.
It remains as open question whether or not there is any nonzero strain matrix $S\in\mathcal{L}^2_{st}$ such that $\lambda_2(x)=\lambda_3(x)$ for all $x\in\mathbb{R}^3$. 

\begin{definition}
We will define the space of max-mid matrices to be the trace free matrices with the desired eigenvalue structure,
\begin{equation}
    \mathcal{L}^2_{maxmid}=
    \left\{M\in\mathcal{L}^2:
    \tr(M)=0, \lambda_2=\lambda_3\right\}.
\end{equation}
\end{definition}

\begin{conjecture} \label{GapConjecture}
There are no nontrivial max-mid matrices in the strain constraint space:
\begin{equation}
    \mathcal{L}^2_{st}\cap \mathcal{L}^2_{maxmid}
    =\{0\}.
\end{equation}
Furthermore, we can establish a gap between these spaces with
\begin{equation}
    \sup_{\substack{M\in\mathcal{L}^2_{maxmid}\\ 
    \|M\|_{L^2}=1
    }}
    \left\|P_{st}(M)\right\|_{L^2}^2
    =
    \sup_{\substack{\|\lambda\|_{L^2}=1 \\ 
    \lambda\geq 0\\
    v\in\mathbf{L}^\infty \\
    |v(x)|=1
    }}
    \left\|P_{st}\left(\frac{\lambda}{\sqrt{6}}
    \left(I_3-3v \otimes v\right)\right)
    \right\|_{L^2}^2
    <1.
\end{equation}
\end{conjecture}

\begin{remark}
We will show in section \ref{InequalitiesSection} that the two suprema are equal with the latter serving as a parameterization of $\mathcal{L}^2_{maxmid}$.
The normalization factor of $\frac{1}{\sqrt{6}}$ comes from the fact that
\begin{equation}
    \left\|
    \frac{\lambda}{\sqrt{6}}
    \left(I_3-3v \otimes v\right)
    \right\|_{L^2}^2
    =
    \|\lambda\|_{L^2}^2.
\end{equation}
\end{remark}

\begin{remark}
Conjecture \ref{GapConjecture} is very important for the Navier--Stokes regularity problem 
because any nontrivial matrix in
$\mathcal{L}^2_{st}\cap \mathcal{L}^2_{maxmid}$,
would form a very natural candidate for finite-time blowup. If Conjecture \ref{GapConjecture} fails, even if the supremum of one is not attained, then the near maximizers would be very natural blowup candidates. 
Likewise, if Conjecture \ref{GapConjecture} holds, then the near maximizers (or maximizers if they exist) would still form natural candidates for possible finite-time blowup; however, in this case, the structure in the constraint space that requires the supremum to be less than one would also be worth investigating as a source of cancellation in the growth of enstrophy which might lead to global regularity. 

It should be noted that, while it remains an open question whether a strain matrix may satisfy $\lambda_2=\lambda_3$ globally in space, this geometric structure already features heavily in the existing literature on finite-time blowup for incompressible fluids. This $-2\lambda, \lambda,\lambda$ eigenvalue structure is precisely the geometric structure of the strain matrix at the stagnation point blowup in Elingdi's recent proof of finite-time blowup for $C^{1,\alpha}$ solutions of the incompressible Euler equation \cite{Elgindi}.
\end{remark}

We will not be able to resolve the question posed in Conjecture \ref{GapConjecture} in this paper, however we can compute the supremum in the special case where the unit vector $v\in\mathbb{R}^3, |v|=1$ is fixed, giving both the supremum explicitly and a family of near maximizers. The fixed direction case allows us to use information from the Fourier space side characterization of $\mathcal{L}^2_{st}$ that is not available in general for matrices in $\mathcal{L}^2_{maxmid}$.
We will discuss the barriers to the general case in more detail in section \ref{InequalitiesSection}.

\begin{theorem} \label{OneDirectionMaxIntro}
Let $v\in\mathbb{R}^3$ be a fixed unit vector, $|v|=1$, then purely axial compression in the $v$, and planar stretching in the plane perpendicular to $v$ is not possible. That is there does not exist $\lambda \in L^2, \lambda\geq 0, \lambda$ not identically zero such that
\begin{equation}
    \lambda \left(I_3-v\otimes v\right)
    \in \mathcal{L}^2_{st}.
\end{equation}
In fact, for all $v\in\mathbb{R}^3, |v|=1$ we have the explicit bound on the projection of matrices of the form onto $\mathcal{L}^2_{st}$,
\begin{equation}
    \sup_{\substack{\|\lambda\|_{L^2}=1 \\
    \lambda \geq 0}}
    \left\|P_{st}\left(\frac{\lambda}{\sqrt{6}}
    \left(I_3-3v \otimes v\right)\right)
    \right\|_{L^2}^2=\frac{3}{4}.
\end{equation}

We can also give explicitly a family of near maximizers for the case $v=e_3$ as follows. Suppose 
$\|\lambda\|_{L^2}=1,$ and
$\supp\left(\hat{\lambda}\right) \in 
\left\{\xi\in\mathbb{R}^3: 
\frac{1}{2}-\frac{\epsilon}{2}
<\frac{\xi_3^2}{|\xi|^2}
<\frac{1}{2}+\frac{\epsilon}{2}\right\}$.
Then
\begin{equation}
    \left\|P_{st}\left(\frac{\lambda}{\sqrt{6}}
    \left(I_3-3 e_3\otimes e_3\right)\right)\right\|_{L^2}
    >\frac{3}{4}(1-\epsilon)^2.
\end{equation}
The analogous family of maximizers for any other fixed $v\in\mathbb{R}^3, |v|=1$ can be immediately deduced by the rotational symmetry of $\mathcal{L}^2_{st}$ given in Proposition \ref{RotInvarIntro}.
\end{theorem}

Furthermore, we can also establish that if the direction of axial compression can vary in a fixed plane, but not all directions, then this is enough to rule out nontrivial maxmid, strain matrices, although in this case we cannot establish a gap.

\begin{theorem} \label{TwoDirectionCaseIntro}
Suppose that $|v(x)|=1$ and $v_1(x)=0$ almost everywhere $x\in\mathbb{R}^3$. That is, suppose $v$ is a unit vector varying in space, but restricted to the $yz$ plane.
Then there does not exist $\lambda\in L^2, \lambda\geq 0,$ and $\lambda$ not identically zero, such that
\begin{equation}
    \lambda(I_3-3v\otimes v)\in \mathcal{L}^2_{st}.
\end{equation}
This means that any nontrivial maxmid strain matrix must compress in all three directions. The axis of compression cannot be confined to any fixed plane.
\end{theorem}

While we cannot resolve Conjecture \ref{GapConjecture} and establish a gap in the general case, we can reduce this question to the projection of rank one matrices onto the strain space as follows:
\begin{equation} \label{RankOneReduction}
    \sup_{\substack{\|\lambda\|_{L^2}=1 \\ 
    \lambda\geq 0\\
    |v(x)|=1
    }}
    \left\|P_{st}\left(\frac{\lambda}{\sqrt{6}}
    \left(I_3-3v \otimes v\right)\right)
    \right\|_{L^2}^2
    =
    \frac{3}{2}\sup_{\|w\|_{L^4}=1}
    \left\|P_{st}(w \otimes w)\right\|_{L^2}^2.
\end{equation}

Finally, we will prove that if Conjecture \ref{GapConjecture} holds, then we have an inequality giving a lower bound on the size of $\lambda_3-\lambda_2^+$ for all matrices in $\mathcal{L}^2_{st}$.

\begin{theorem} \label{EigenGapIntro}
Suppose
\begin{equation} 
    \sup_{\substack{\|\lambda\|_{L^2}=1 \\ 
    \lambda\geq 0\\
    |v(x)|=1
    }}
    \left\|P_{st}\left(\frac{\lambda}{\sqrt{6}}
    \left(I_3-3v \otimes v\right)\right)
    \right\|_{L^2}^2=r<1.
\end{equation}
Then for all $S\in\mathcal{L}^2_{st}$,
\begin{equation}
    \left\|\lambda_3-\lambda_2^+\right\|_{L^2}\geq 
    \frac{1-r}{\sqrt{2}} \|S\|_{L^2},
\end{equation}
where $\lambda_1(x)\leq \lambda_2(s)\leq \lambda_3(x)$
are the eigenvalues of $S(x)$ 
and $\lambda_2^+=\max(0,\lambda_2)$.
\end{theorem}

Another useful relation involving the strain projection---which is also germane to the Navier--Stokes regularity problem---is a commutator-type relationship involving the divergence and the classical Helmholtz projection given by
\begin{equation} \label{ProjectionDiv}
    \divr P_{st}=P_{df} \divr.
\end{equation}
This relationship will allow us to derive a viscous Hamilton-Jacobi type equation for the Matrix potential for the Navier--Stokes equation.
If we let $M=2(-\Delta)^{-1}S,$ then we will also have
\begin{equation}
    u=-\divr(M),
\end{equation}
so $M$ is the matrix potential for $u$.
The Navier--Stokes equation can be expressed in terms of the Helmholtz projection as
\begin{equation} \label{NS}
    \partial_t u-\Delta u+P_{df}\divr (u\otimes u)=0.
\end{equation}
We will show using the identity \eqref{ProjectionDiv}, that $u$ is a solution of the Navier--Stokes equation \eqref{NS}
if and only if
$M$ is a solution of the Navier--Stokes matrix potential equation
\begin{equation} \label{MatrixPotentialIntro}
    \partial_t M-\Delta M-P_{st}
    \left(\divr(M)\otimes\divr(M)\right)=0.
\end{equation}
This is particularly interesting because the equation 
\eqref{MatrixPotentialIntro} has a viscous Hamilton-Jacobi type structure, and the viscous Hamilton-Jacobi equation
\begin{equation}
    \partial_t f-\Delta f-|\nabla f|^2=0,
\end{equation}
has global smooth solutions. There is no clear way to see that the Navier--Stokes matrix potential equation has global smooth solutions, because the projection involved is a Riesz type transform and because $\divr(M)\otimes\divr(M)$ is a much more complicated quadratic term than $|\nabla f|^2$ for a scalar function $f$.
Nonetheless, the matrix potential equation, introduced here for the first time, deserves further study.
We will discuss this in detail in section \ref{MatrixPotentialSection}.

We will also consider the orthogonal decomposition of square integrable symmetric matrices for the general $d$-dimensional case, giving the decomposition for
$\mathcal{L}^2\left(\mathbb{R}^d\right)$.
When $d=2$, we will show that
\begin{equation}
\mathcal{L}^2\left(\mathbb{R}^2\right)
=\mathcal{L}^2_{st}\left(\mathbb{R}^2\right)
\oplus \mathcal{L}^2_{Hess}\left(\mathbb{R}^2\right)
\oplus \mathcal{L}^2_{\widetilde{Id}}\left(\mathbb{R}^2\right).
\end{equation}
This is true in the two dimensional case 
in particular because
\begin{equation}
    \mathcal{L}^2_{tr\&divfree}\left(\mathbb{R}^2\right)=
    \ker(\tr)\cap\ker(\divr)
    =\left\{\mathbf{0}\right\}.
\end{equation}
We will also show that in the general $d$ dimensional case with $d\geq 3$, we have the orthogonal decomposition
\begin{align}
    \mathcal{L}^2 \left(\mathbb{R}^d\right)
    &=
    \mathcal{L}^2_{st}\left(\mathbb{R}^d\right)
    \oplus \mathcal{L}^2_{Hess}\left(\mathbb{R}^d\right)
    \oplus \mathcal{L}^2_{divfree}
    \left(\mathbb{R}^d\right) \\
    &=
    \mathcal{L}^2_{st}\left(\mathbb{R}^d\right)
    \oplus \mathcal{L}^2_{Hess}\left(\mathbb{R}^d\right)
    \oplus \mathcal{L}^2_{\widetilde{Id}}\left(\mathbb{R}^d\right)
    \oplus \mathcal{L}^2_{tr\&divfree}
    \left(\mathbb{R}^d\right).
\end{align}
While we do have the same decomposition 
for $d> 3$, as for $d=3$, there are significant geometric differences buried in the structure of these spaces, and that for larger dimensions, the space
$\mathcal{L}^2_{st}\left(\mathbb{R}^d\right)$ 
will have more degrees of freedom, and the space 
$\mathcal{L}^2_{divfree}\left(\mathbb{R}^d\right)$
will have substantially more degrees of freedom,
so the structure does in a real way depend on dimension, because the larger the dimension of the domain $\mathbb{R}^d$,
the larger the dimension of the kernel of divergence on the Fourier space side.

Finally, we will discuss the analog of Theorem \ref{HelmholtzOneIntro} for anti-symmetric matrices.
It follows immediately from elementary linear algebra, that matrix valued $L^2$ functions have a decomposition into symmetric and anti-symmetric matrices:
\begin{equation}
    L^2\left(\mathbb{R}^d;\mathbb{R}^{d\times d}\right)
    =
    L^2\left(\mathbb{R}^d;\mathcal{S}^{d\times d}\right)
    \oplus
    L^2\left(\mathbb{R}^d;\mathcal{A}^{d\times d}\right).
\end{equation}
Similar to the Helmholtz decomposition for symmetric matrices, we can also give an orthogonal decomposition for anti-symmetric matrices. First we will define two new spaces.

\begin{definition}
We define the subspace of anti-symmetric gradients of divergence free vector fields, as the vorticity matrices,
\begin{equation}
    \mathbb{L}^2_{vort}
    =\left\{
    \nabla_{asym}(-\Delta)^{-\frac{1}{2}}v:
    v\in \mathbf{L}^2_{df}.
    \right\}
\end{equation}
We define the subspace of divergence free anti-symmetric matrices by
\begin{equation}
    \mathbb{L}^2_{divfree}
    =
    \ker(\divr)
    =
    \left\{
    M\in L^2\left(\mathbb{R}^d;\mathcal{A}^{d\times d}\right)
    :\divr(-\Delta)^{-\frac{1}{2}}M=0\right\}.
\end{equation}
\end{definition}

\begin{remark}
Note that in general the vorticity is expressed as an anti-symmetric matrix. It is only in the special case $d=3$ that the vorticity can be expressed as a vector, and only in the special case $d=2$ that it can be expressed as a scalar.
\end{remark}

\begin{theorem} \label{HelmholtzAntiSymIntro}
For all $d\geq 2$, the space of anti-symmetric matrix valued functions 
$L^2\left(\mathbb{R}^d;A^{d\times d}\right)$ has the following orthogonal decomposition:
\begin{equation}
    L^2\left(\mathbb{R}^d;\mathcal{A}^{d\times d}\right)
    =
    \mathbb{L}^2_{vort} \oplus \mathbb{L}^2_{divfree}. 
\end{equation}
Furthermore, this yields the general orthogonal decomposition of matrix valued functions,
\begin{align}
    L^2\left(\mathbb{R}^d;\mathbb{R}^{d\times d}\right)
    &=
    L^2\left(\mathbb{R}^d;\mathcal{S}^{d\times d}\right)
    \oplus
    L^2\left(\mathbb{R}^d;\mathcal{A}^{d\times d}\right) \\
    &=
    \mathcal{L}^2_{st} \oplus \mathcal{L}^2_{Hess} 
\oplus \mathcal{L}^2_{divfree} \oplus
    \mathbb{L}^2_{vort} \oplus \mathbb{L}^2_{divfree}.
\end{align}
\end{theorem}

We will also note that in the special case $d=3$, the spaces $\mathbb{L}^2_{vort}$ and $\mathbb{L}^2_{divfree}$ can be parameterized by divergence free vector fields and gradients respectively, giving an isometry with the classical Helmholtz decomposition, Theorem \ref{HelmholtzOld}.

\section{The orthogonal decomposition}
\label{DecompositionSection}

In this section, we will derive the orthogonal decomposition for symmetric matrices in three dimensions.
We will begin with some Fourier analysis of the respective spaces, as the proof will be done entirely in Fourier space.

\begin{proposition} \label{FourierConstraint} 
For all $\xi \in \mathbb{R}^{3},\xi \neq 0,$ take $
\eta ,\mu \in \mathbb{R}^{3}$ to be unit vectors such that 
\begin{equation}
\mathcal{B}\left( \xi \right) =\left\{ \frac{\xi }{\left\vert \xi
\right\vert },\eta ,\mu \right\}
\end{equation}
is an orthonormal basis for $\mathbb{R}^{3}$.
Then for all $M\in\mathcal{L}^2$: 

\begin{itemize}

\item
$M\in \mathcal{L}_{st}^{2}$ if and only if
the representation of $\widehat{M}(\xi)$ in the basis $\mathcal{B}\left( \xi \right)$ satisfies
\begin{equation}
\left[ \widehat{M}\left( \xi \right)
\right]_{\mathcal{B}\left( \xi \right)}
\in \spn \left\{ \frac{1}{\sqrt{2}}\left[ 
\begin{array}{ccc}
0 & 1 & 0 \\ 
1 & 0 & 0 \\ 
0 & 0 & 0
\end{array}
\right] ,\frac{1}{\sqrt{2}}\left[ 
\begin{array}{ccc}
0 & 0 & 1 \\ 
0 & 0 & 0 \\ 
1 & 0 & 0
\end{array}
\right] \right\},
\end{equation}
almost everywhere $\xi \in \mathbb{R}^{3}$.

\item
$M\in \mathcal{L}_{Hess}^{2}$ if and only if
the representation of $\widehat{M}(\xi)$ in the basis $\mathcal{B}\left( \xi \right)$ satisfies
\begin{equation}
\left[ \widehat{M}\left( \xi \right)
\right]_{\mathcal{B}\left( \xi \right) }
\in \spn \left\{ \left[ 
\begin{array}{ccc}
1 & 0 & 0 \\ 
0 & 0 & 0 \\ 
0 & 0 & 0
\end{array}
\right] \right\},
\end{equation}
almost everywhere $\xi \in \mathbb{R}^{3}$.

\item
$M\in \mathcal{L}_{divfree}^{2}$ if and only if
the representation of $\widehat{M}(\xi)$ in the basis $\mathcal{B}\left( \xi \right)$ satisfies

\begin{equation}
\left[ \widehat{M}\left( \xi \right)
\right]_{\mathcal{B}\left( \xi \right)}
\in \spn \left\{ \frac{1}{\sqrt{2}}\left[ 
\begin{array}{ccc}
0 & 0 & 0 \\ 
0 & 1 & 0 \\ 
0 & 0 & 1
\end{array}
\right] ,\frac{1}{\sqrt{2}}\left[ 
\begin{array}{ccc}
0 & 0 & 0 \\ 
0 & 1 & 0 \\ 
0 & 0 & -1
\end{array}
\right] ,\frac{1}{\sqrt{2}}\left[ 
\begin{array}{ccc}
0 & 0 & 0 \\ 
0 & 0 & 1 \\ 
0 & 1 & 0
\end{array}
\right] \right\},
\end{equation}
almost everywhere $\xi \in \mathbb{R}^{3}$.
\end{itemize}
\end{proposition}

\begin{proof}
We begin by fixing $M\in \mathcal{L}_{st}^{2},$ 
taking $u\in \mathbf{L}^2_{df} $ such that 
$M=\nabla_{sym}(-\Delta)^{-\frac{1}{2}}u.$ 
First we will need to write the
divergence free constraint in Fourier spaces terms: $\nabla \cdot u=0,$
implies that for all $\xi \in \mathbb{R}^{3},$ $\xi \cdot \hat{u}(\xi )=0.$
Therefore there exists $a,b\in \mathbb{C},$ such that $\hat{u}(\xi )=a\eta
+b\mu .$ This implies that 
\begin{align}
\widehat{(\nabla u)}(\xi )& =2\pi i\xi \otimes \hat{u}(\xi ) \\
& =2\pi ai\xi \otimes \eta +2\pi bi\xi \otimes \mu .
\end{align}
Symmetrizing this expression and 
dividing by $2\pi|\xi|$ we find that 
\begin{equation}
\widehat{M}(\xi )= \frac{ia}{2}\left( \frac{\xi }{|\xi |}\otimes \eta +\eta
\otimes \frac{\xi }{|\xi |}\right) +\frac{ib}{2} \left( \frac{\xi }{|\xi |}
\otimes \mu +\mu \otimes \frac{\xi }{|\xi |}\right) .
\end{equation}
Therefore we can see that 
\begin{equation}
\left[ \widehat{M}\left( \xi \right)
\right]_{\mathcal{B}\left( \xi \right)}
=
\frac{i}{2} \left( 
\begin{array}{ccc}
0 & a & b \\ 
a & 0 & 0 \\ 
b & 0 & 0
\end{array}
\right) ,
\end{equation}
and so we may conclude that 
\begin{equation}
\left[ \widehat{M}\left( \xi \right)
\right]_{\mathcal{B}\left( \xi \right)}
\in \spn
\left\{ \frac{1}{\sqrt{2}}\left( 
\begin{array}{ccc}
0 & 1 & 0 \\ 
1 & 0 & 0 \\ 
0 & 0 & 0
\end{array}
\right) ,\frac{1}{\sqrt{2}}\left( 
\begin{array}{ccc}
0 & 0 & 1 \\ 
0 & 0 & 0 \\ 
1 & 0 & 0
\end{array}
\right) \right\} .
\end{equation}

Now suppose that $\widehat{M}\in \mathcal{L}^{2}$ and 
\begin{equation}
\left[ \widehat{M}\left( \xi \right)
\right]_{\mathcal{B}\left( \xi \right)}
\in \spn
\left\{ \frac{1}{\sqrt{2}}\left( 
\begin{array}{ccc}
0 & 1 & 0 \\ 
1 & 0 & 0 \\ 
0 & 0 & 0
\end{array}
\right) ,\frac{1}{\sqrt{2}}\left( 
\begin{array}{ccc}
0 & 0 & 1 \\ 
0 & 0 & 0 \\ 
1 & 0 & 0
\end{array}
\right) \right\} .
\end{equation}
Following the arguments above we can see that for almost all $\xi \in 
\mathbb{R}^{3},$ 
\begin{equation}
\left[ \widehat{M}\left( \xi \right)
\right]_{\mathcal{B}\left( \xi \right)}
=i\left( 
\begin{array}{ccc}
0 & a & b \\ 
a & 0 & 0 \\ 
b & 0 & 0
\end{array}
\right) ,
\end{equation}
and so we can conclude 
\begin{equation}
\widehat{M}(\xi )=ia \left( \frac{\xi }{|\xi |}\otimes \eta +\eta
\otimes \frac{\xi }{|\xi |}\right) 
+ib \left( \frac{\xi }{|\xi |}
\otimes \mu +\mu \otimes \frac{\xi }{|\xi |}\right) ,
\end{equation}
where $a,b\in L^{2}\left( \mathbb{R}^{3}\right) .$ Letting 
\begin{equation}
\hat{u}(\xi )=2a\eta +2b\mu,
\end{equation}
observe that 
\begin{equation}
\mathcal{F}{(\nabla _{sym}(-\Delta)^{-\frac{1}{2}}u)} (\xi)
= ia\left( \frac{\xi }{|\xi |}
\otimes \eta +\eta \otimes \frac{\xi }{|\xi |}\right) +ib\left( 
\frac{\xi }{|\xi |}\otimes \mu +\mu \otimes \frac{\xi }{|\xi |}\right).
\end{equation}
This implies that 
\begin{equation}
    M=\nabla_{sym}(-\Delta)^{-\frac{1}{2}}u.
\end{equation}
Noting that
\begin{align}
    \left|\hat{u}(\xi)\right|^2
    &=4a^2+4b^2 \\
    &= 2\left|\hat{M}(\xi)\right|^2
\end{align}
and 
\begin{equation}
\xi \cdot \hat{u}(\xi )=0,
\end{equation}
we can conclude that $u\in \mathbf{L}_{df}^{2}$ and so
$M\in \mathcal{L}_{st}^{2}.$

Next we consider $M\in \mathcal{L}_{Hess}^{2},$ taking $M=\Hess(-\Delta
)^{-1}f.$ Taking the Fourier transform we find 
\begin{align}
\widehat{M}(\xi )& =2\pi i\xi \otimes 2\pi i\xi \frac{1}{4\pi ^{2}|\xi |^{2}}
\hat{f}(\xi ) \\
& =-\frac{\xi }{|\xi |}\otimes \frac{\xi }{|\xi |}\hat{f}(\xi ).
\end{align}
Therefore we can see that 
\begin{equation}
\left[ \widehat{M}\left( \xi \right)
\right]_{\mathcal{B}\left( \xi \right)}
=-\hat{f}
(\xi )\left( 
\begin{array}{ccc}
1 & 0 & 0 \\ 
0 & 0 & 0 \\ 
0 & 0 & 0
\end{array}
\right) ,
\end{equation}
and so we may conclude that 
\begin{equation}
\left[ \widehat{M}\left( \xi \right)
\right]_{\mathcal{B}\left( \xi \right)}
\in \spn
\left\{ \left( 
\begin{array}{ccc}
1 & 0 & 0 \\ 
0 & 0 & 0 \\ 
0 & 0 & 0
\end{array}
\right) \right\} .
\end{equation}

Now consider $\widehat{M}\in \mathcal{L}^{2},$ 
\begin{equation}
\left[ \widehat{M}\left( \xi \right)
\right]_{\mathcal{B}\left( \xi \right)}
\in \spn
\left\{ \left( 
\begin{array}{ccc}
1 & 0 & 0 \\ 
0 & 0 & 0 \\ 
0 & 0 & 0
\end{array}
\right) \right\} .
\end{equation}
Then there is an $a\in L^{2},$ such that 
\begin{equation}
\widehat{M}(\xi )=-a\frac{\xi }{|\xi |}\otimes \frac{\xi }{|\xi |}.
\end{equation}
Taking the inverse Fourier transform of f, we set 
\begin{equation}
\hat{f}=a,
\end{equation}
so 
\begin{equation*}
\hat{M}(\xi )=-\hat{f}(\xi )\frac{\xi }{|\xi |}\otimes \frac{\xi }{|\xi |}.
\end{equation*}
Therefore we can conclude that 
\begin{equation}
M=\Hess(-\Delta )^{-1}f,
\end{equation}
and so $M\in \mathcal{L}_{Hess}^{2}$

Finally we consider $M\in \mathcal{L}_{divfree}^{2}.$ We first observe that
for all $j\in \{1,2,3\},$ 
\begin{equation}
\divr(M)_{j}=\sum_{k=1}^{3}\partial _{k}M_{jk},
\end{equation}
noting that the row and column divergence of a matrix is equivalent for symmetric matrices.
Taking the Fourier transform of this identity we find that 
\begin{equation}
\widehat{\divr(M)}_{j}=2\pi i\sum_{k=1}^{3}\widehat{M}_{jk}\xi _{k},
\end{equation}
so we can conclude that 
\begin{equation*}
\widehat{\divr(M)}=2\pi i\widehat{M}\xi .
\end{equation*}
Therefore the condition $\divr(M)=0,$ can be restated in Fourier space as $
\widehat{M}\xi =0,$ for all $\xi \in \mathbb{R}^{3}.$ This implies that the
first column, and hence by symmetry the first row, of $\left[ \widehat{M}(\xi )
\right]_{\mathcal{B}\left(\xi\right)}$
is zero. Therefore we may conclude that 
\begin{equation}
\left[ \widehat{M}\left( \xi \right)
\right]_{\mathcal{B}\left( \xi \right)} 
=\left( 
\begin{array}{ccc}
0 & 0 & 0 \\ 
0 & a & c \\ 
0 & c & b
\end{array}
\right).
\end{equation}
Note that 
\begin{equation}
\left( 
\begin{array}{ccc}
0 & 0 & 0 \\ 
0 & a & c \\ 
0 & c & b
\end{array}
\right) =\frac{a+b}{2}\left( 
\begin{array}{ccc}
0 & 0 & 0 \\ 
0 & 1 & 0 \\ 
0 & 0 & 1
\end{array}
\right) +\frac{a-b}{2}\left( 
\begin{array}{ccc}
0 & 0 & 0 \\ 
0 & 1 & 0 \\ 
0 & 0 & -1
\end{array}
\right) +c\left( 
\begin{array}{ccc}
0 & 0 & 0 \\ 
0 & 0 & 1 \\ 
0 & 1 & 0
\end{array}
\right) ,
\end{equation}
and we may conclude that 
\begin{equation}
\left[ \widehat{M}\left( \xi \right)
\right]_{\mathcal{B}\left( \xi \right)}
\in \spn
\left\{ \frac{1}{\sqrt{2}}\left( 
\begin{array}{ccc}
0 & 0 & 0 \\ 
0 & 1 & 0 \\ 
0 & 0 & 1
\end{array}
\right) ,\frac{1}{\sqrt{2}}\left( 
\begin{array}{ccc}
0 & 0 & 0 \\ 
0 & 1 & 0 \\ 
0 & 0 & -1
\end{array}
\right) ,\frac{1}{\sqrt{2}}\left( 
\begin{array}{ccc}
0 & 0 & 0 \\ 
0 & 0 & 1 \\ 
0 & 1 & 0
\end{array}
\right) \right\} .
\end{equation}

Now suppose that $\hat{M}\in \mathcal{L}^{2}$ and 
\begin{equation}
\left[ \widehat{M}\left( \xi \right)
\right]_{\mathcal{B}\left( \xi \right)}
\in \spn
\left\{ \frac{1}{\sqrt{2}}\left( 
\begin{array}{ccc}
0 & 0 & 0 \\ 
0 & 1 & 0 \\ 
0 & 0 & 1
\end{array}
\right) ,\frac{1}{\sqrt{2}}\left( 
\begin{array}{ccc}
0 & 0 & 0 \\ 
0 & 1 & 0 \\ 
0 & 0 & -1
\end{array}
\right) ,\frac{1}{\sqrt{2}}\left( 
\begin{array}{ccc}
0 & 0 & 0 \\ 
0 & 0 & 1 \\ 
0 & 1 & 0
\end{array}
\right) \right\} .
\end{equation}
We can clearly see that for almost all $\xi \in \mathbb{R}^{3}$ 
\begin{equation}
\left[ \widehat{M}\left( \xi \right)
\right]_{\mathcal{B}\left( \xi \right)}
e_{1}=0,
\end{equation}
which implies that 
\begin{equation}
\widehat{M}(\xi )\xi =0.
\end{equation}
As we have shown above, this is the exactly the condition $\divr(M)=0,$
written in Fourier space terms, so we can conclude that $M\in \mathcal{L}_{divfree}^{2}.$ This completes the proof.
\end{proof}

\begin{lemma}
\label{LemmaSym6} The set of matrices 
\begin{equation}
\begin{split}
\mathcal{B}= \Bigg \{ \left(
\begin{array}{ccc}
1 & 0 & 0 \\ 
0 & 0 & 0 \\ 
0 & 0 & 0
\end{array}
\right), \frac{1}{\sqrt{2}} \left(
\begin{array}{ccc}
0 & 1 & 0 \\ 
1 & 0 & 0 \\ 
0 & 0 & 0
\end{array}
\right), \frac{1}{\sqrt{2}} \left(
\begin{array}{ccc}
0 & 0 & 1 \\ 
0 & 0 & 0 \\ 
1 & 0 & 0
\end{array}
\right), \\
\frac{1}{\sqrt{2}} \left(
\begin{array}{ccc}
0 & 0 & 0 \\ 
0 & 1 & 0 \\ 
0 & 0 & 1
\end{array}
\right), \frac{1}{\sqrt{2}} \left(
\begin{array}{ccc}
0 & 0 & 0 \\ 
0 & 1 & 0 \\ 
0 & 0 & -1
\end{array}
\right), \frac{1}{\sqrt{2}} \left(
\begin{array}{ccc}
0 & 0 & 0 \\ 
0 & 0 & 1 \\ 
0 & 1 & 0
\end{array}
\right) \Bigg\}
\end{split}
\end{equation}
is an orthonormal basis for the space of symmetric matrices, 
\begin{equation}
\mathcal{S}^{3\times 3}_{\mathbb{C}}
=\left\{ \left(
\begin{array}{ccc}
a & d & e \\ 
d & b & f \\ 
e & f & c
\end{array}
\right): a,b,c,d,e,f \in \mathbb{C} \right\}.
\end{equation}
\end{lemma}

\begin{proof}
The proof is a trivial linear algebra exercise. It is easy to check that the matrices in $\mathcal{B}$ are orthonormal, and it is clear that $\dim\left(
\mathcal{S}^{3\times 3}\right)=6$. This means any set of $6$ orthonormal
matrices is an orthonormal basis, and so this completes the proof.
\end{proof}

Now that the preliminary Fourier analysis is out of the way, we will prove Theorem \ref{HelmholtzOneIntro}, which is restated for the reader's convenience.

\begin{theorem} \label{HelmholtzOne}
We have the orthogonal decomposition
\begin{equation}
\mathcal{L}^2=\mathcal{L}^2_{st} \oplus \mathcal{L}^2_{Hess} \oplus 
\mathcal{L}^2_{divfree}
\end{equation}
\end{theorem}

\begin{proof}
We will start by proving that 
\begin{equation}
\mathcal{L}^{2}=\mathcal{L}_{st}^{2}+\mathcal{L}_{Hess}^{2}+\mathcal{L}
_{divfree}^{2}.
\end{equation}
Fix $M\in \mathcal{L}^{2}.$ Let 
\begin{align}
V_{1}& =\spn\left\{ \frac{1}{\sqrt{2}}\left( 
\begin{array}{ccc}
0 & 1 & 0 \\ 
1 & 0 & 0 \\ 
0 & 0 & 0
\end{array}
\right) ,\frac{1}{\sqrt{2}}\left( 
\begin{array}{ccc}
0 & 0 & 1 \\ 
0 & 0 & 0 \\ 
1 & 0 & 0
\end{array}
\right) \right\} , \\
V_{2}& =\spn\left\{ \left( 
\begin{array}{ccc}
1 & 0 & 0 \\ 
0 & 0 & 0 \\ 
0 & 0 & 0
\end{array}
\right) \right\} , \\
V_{3}& =\spn\left\{ \frac{1}{\sqrt{2}}\left( 
\begin{array}{ccc}
0 & 0 & 0 \\ 
0 & 1 & 0 \\ 
0 & 0 & 1
\end{array}
\right) ,\frac{1}{\sqrt{2}}\left( 
\begin{array}{ccc}
0 & 0 & 0 \\ 
0 & 1 & 0 \\ 
0 & 0 & -1
\end{array}
\right) ,\frac{1}{\sqrt{2}}\left( 
\begin{array}{ccc}
0 & 0 & 0 \\ 
0 & 0 & 1 \\ 
0 & 1 & 0
\end{array}
\right) \right\}.
\end{align}
From Lemma \ref{LemmaSym6}, we know that 
\begin{equation}
\mathcal{S}^{3 \times 3}
=V_{1}\oplus V_{2}\oplus V_{3}.
\end{equation}
Therefore for each $\xi \in \mathbb{R}^{3},$ there exists a unique $\Tilde{M}^{1}(\xi )\in
V_{1},\Tilde{M}^{2}(\xi )\in V^{2},\Tilde{M}^{3}(\xi )\in V_{3},$ such that 
\begin{equation}
\left[ \widehat{M}(\xi)
\right]_{\mathcal{B}(\xi)} 
=\Tilde{M}^{1}+\Tilde{M}^{2}+
\Tilde{M}^{3}
\end{equation}
Inverting the Fourier transform,
take $M^{1},M^{2},M^{3}\in\mathcal{L}^2$ such that 
\begin{align}
\left[ \widehat{M}^{1}\right] _{\mathcal{B}(\xi )}& =\Tilde{M^{1}} \\
\left[ \widehat{M}^{2}\right] _{\mathcal{B}(\xi )}& =\Tilde{M^{2}} \\
\left[ \widehat{M}^{3}\right] _{\mathcal{B}(\xi )}& =\Tilde{M^{3}}.
\end{align}
Clearly we will then have
\begin{equation}
M=M^{1}+M^{2}+M^{3},
\end{equation}
and by Proposition \ref{FourierConstraint}, we can see that $M^{1}\in 
\mathcal{L}_{st}^{2},M^{2}\in \mathcal{L}_{Hes}^{2},M^{3}\in \mathcal{L}
_{divfree}^{2}.$ Therefore we can conclude that 
\begin{equation}
\mathcal{L}^{2}=\mathcal{L}_{st}^{2}+\mathcal{L}_{Hess}^{2}+\mathcal{L}
_{divfree}^{2}.
\end{equation}

It remains to prove orthogonality. Using the fact that the Fourier transform
is an isometry, and that the matrix inner product is independent of the
choice of orthonormal basis, we compute that for $M,S\in \mathcal{L}^{2}$ 
\begin{align}
\left\langle S,M\right\rangle & =\left\langle \widehat{S},\widehat{M}
\right\rangle \\
& =\int_{\mathbb{R}^{3}}
\left[ \widehat{M}(\xi )
\right]_{\mathcal{B}(\xi }
\odot \left[ \bar{\widehat{S}}(\xi )
\right]_{\mathcal{B}(\xi )} \diff\xi.
\end{align}
This implies that the orthogonality of $V_{1},V_{2},V_{3},$ together with
Proposition \ref{FourierConstraint}, which states the image of the Fourier
transform with respect to the basis $\mathcal{B}(\xi )$ must be in $
V_{1},V_{2},V_{3}$ for $\mathcal{L}_{st}^{2},\mathcal{L}_{Hess}^{2},
\mathcal{L}_{divfree}^{2}$ respectively, implies that $\mathcal{L}
_{st}^{2},\mathcal{L}_{Hess}^{2},\mathcal{L}_{divfree}^{2}$ are
orthogonal. This completes the proof.
\end{proof}

\begin{remark}
While Theorem \ref{HelmholtzOne} is now proven, most of the differential structure of the spaces has been hidden in the way we picked our basis for the Fourier transform for each $\xi\in\mathbb{R}^3$,
so it is worthwhile to go through the 
formal computations showing that 
$\mathcal{L}^2_{st}, \mathcal{L}^2_{Hess},
\mathcal{L}^2_{divfree}$
are orthogonal. The proof that their sum space is all of $\mathcal{L}^2$ is much harder on the physical space side, because we cannot simply count dimensions, 
which is all the proof of the sum space result amounts to on the Fourier space side, once the choice of
basis $\mathcal{B}(\xi)$ reduces the problem to a six dimensional linear algebra problem. 
The orthogonality, on the other hand, can be understood quite easily by integrating by parts formally, which we will do now.

Fix $S\in \mathcal{L}_{st}^{2},M\in \mathcal{L}_{divfree}^{2},$ and let $
S=\nabla _{sym}(-\Delta)^{-\frac{1}{2}}u,$ with $u\in \mathbf{L}^2_{df}.$ Then we can compute that 
\begin{align}
\left\langle S,M\right\rangle & =\left\langle \nabla _{sym}(-\Delta)^{-\frac{1}{2}}u,M\right\rangle
\\
&=\left\langle \nabla (-\Delta)^{-\frac{1}{2}} u
,M\right\rangle \\
&=\left\langle u,-\divr(-\Delta)^{-\frac{1}{2}} M\right\rangle \\
& =0.
\end{align}
Fix $S\in \mathcal{L}_{st}^{2},M\in \mathcal{L}_{Hess}^{2},$ 
and let 
$S=\nabla_{sym}(-\Delta)^{-\frac{1}{2}}u,M=\Hess(-\Delta)^{-1}f.$ Then we compute that 
\begin{align}
\left\langle S,M\right\rangle & =\left\langle 
\nabla_{sym}(-\Delta)^{-\frac{1}{2}}u,
\Hess(-\Delta)^{-1}f\right\rangle \\
&=\left\langle \nabla(-\Delta)^{-\frac{1}{2}}u,
\Hess(-\Delta)^{-1}f\right\rangle \\
&=\left\langle u,-\divr\Hess(-\Delta)^{-\frac{3}{2}}f
\right\rangle \\
&=\left\langle u,\nabla (-\Delta)^{-\frac{1}{2}} f\right\rangle \\
&=-\left\langle \nabla \cdot (-\Delta)^{-\frac{1}{2}} u,
f\right\rangle \\
&=0.
\end{align}
Finally fix $M\in \mathcal{L}_{divfree}^{2},Q\in \mathcal{L}_{Hess}^{2},$
and let $Q=\Hess(-\Delta )^{-1}f.$ Then we compute that 
\begin{align}
\left\langle Q,M\right\rangle & =\left\langle 
\Hess(-\Delta)^{-1}f,M\right\rangle \\
&=\left\langle \nabla \nabla (-\Delta )^{-1}f,M\right\rangle \\
&=\left\langle \nabla (-\Delta)^{-\frac{1}{2}}f
,-\divr(-\Delta)^{-\frac{1}{2}} M\right\rangle \\
&= \left\langle f
,\divr^2(-\Delta)^{-1} M\right\rangle \\
&=0.
\end{align}
\end{remark}

In order to proceed with the further decomposition in Theorem \ref{HelmholtzTwoIntro}, we also need to give Fourier space side characterization of $L^2_{\widetilde{Id}}$ and $L^2_{tr\&divfree}$.

\begin{proposition} \label{FourierConstrainDivFree}
For all $\xi \in \mathbb{R}^{3},\xi \neq 0,$ take $
\eta ,\mu \in \mathbb{R}^{3}$ to be unit vectors such that 
\begin{equation}
\mathcal{B}\left( \xi \right) =\left\{ \frac{\xi }{\left\vert \xi
\right\vert },\eta ,\mu \right\}
\end{equation}
is an orthonormal basis for $\mathbb{R}^{3}$.
$M\in \mathcal{L}_{\widetilde{Id}}^{2},$ if and only if $\widehat{M}\in \mathcal{L}
^{2}$ and almost everywhere $\xi \in \mathbb{R}^{3}$ 
\begin{equation}
\left[ \widehat{M}\left( \xi \right)
\right]_{\mathcal{B}\left( \xi \right)}
\in \spn \left\{ \frac{1}{\sqrt{2}}\left[ 
\begin{array}{ccc}
0 & 0 & 0 \\ 
0 & 1 & 0 \\ 
0 & 0 & 1
\end{array}
\right] \right\}.
\end{equation}
$M\in \mathcal{L}_{tr\&divfree}^{2},$ if and only if $\widehat{M}\in \mathcal{L}
^{2}$ and almost everywhere $\xi \in \mathbb{R}^{3}$ 
\begin{equation}
\left[ \widehat{M}\left( \xi \right)
\right]_{\mathcal{B}\left( \xi \right)}
\in \spn \left\{ \frac{1}{\sqrt{2}}\left[ 
\begin{array}{ccc}
0 & 0 & 0 \\ 
0 & 1 & 0 \\ 
0 & 0 & -1
\end{array}
\right] ,\frac{1}{\sqrt{2}}\left[ 
\begin{array}{ccc}
0 & 0 & 0 \\ 
0 & 0 & 1 \\ 
0 & 1 & 0
\end{array}
\right] \right\}.
\end{equation}
\end{proposition}

\begin{proof}
Suppose $M\in \mathcal{L}^2_{\widetilde{Id}}.$ Then there exists $f\in L^2$, such that
\begin{equation}
    M=\Hess(-\Delta)^{-1}f+f I_3.
\end{equation}
Taking the Fourier transform of $M$, we find that
\begin{equation}
    \widehat{M}=\hat{f}\left(I_3
    -\frac{\xi \otimes \xi}{|\xi|^2}\right),
\end{equation}
and therefore we can see that 
\begin{equation}
\left[ \widehat{M}\left( \xi \right)
\right]_{\mathcal{B}\left( \xi \right)}
= \hat{f}(\xi)\left[ 
\begin{array}{ccc}
0 & 0 & 0 \\ 
0 & 1 & 0 \\ 
0 & 0 & 1
\end{array}
\right],
\end{equation}
and
\begin{equation}
\left[ \widehat{M}\left( \xi \right)
\right]_{\mathcal{B}\left( \xi \right)}
\in \spn \left\{ \frac{1}{\sqrt{2}}\left[ 
\begin{array}{ccc}
0 & 0 & 0 \\ 
0 & 1 & 0 \\ 
0 & 0 & 1
\end{array}
\right] \right\}.
\end{equation}
We can also clearly see that $f\in L^2,$ implies that 
$\widehat{M}\in \mathcal{L}^2$, so this completes the proof for the forward direction.

Now suppose that $\widehat{M}\in \mathcal{L}^2$ and for almost all $\xi\in\mathbb{R}^3$,
\begin{equation}
\left[ \widehat{M}\left( \xi \right)
\right]_{\mathcal{B}\left( \xi \right)}
\in \spn \left\{ \frac{1}{\sqrt{2}}\left[ 
\begin{array}{ccc}
0 & 0 & 0 \\ 
0 & 1 & 0 \\ 
0 & 0 & 1
\end{array}
\right] \right\}.
\end{equation}
Then there exists $\hat{f}\in L^2$ such that almost everywhere $\xi\in \mathbb{R}^3,$
\begin{equation}
    \widehat{M}(\xi)=\hat{f}(\xi) \left(I_3
    -\frac{\xi\otimes\xi}{|\xi|^2}\right).
\end{equation}
Taking the inverse Fourier transform, we can conclude that
\begin{equation}
    M=fI_3+\Hess(-\Delta)^{-1}f.
\end{equation}
This completes the proof of the Fourier space characterization of 
$\mathcal{L}^2_{\widetilde{Id}}$.

We will now consider the corresponding Fourier space characterization of $\mathcal{L}^2_{tr\&divfree}$.
Suppose $M\in \mathcal{L}^2_{tr\&divfree}$.
We know from the proof of Proposition \ref{FourierConstraint}
that if $\divr(-\Delta)^{-\frac{1}{2}}M=0$, then
\begin{equation}
    \widehat{M}(\xi)\xi=0
\end{equation}
for almost all $\xi \in \mathbb{R}^3$.
Therefore, using our orthonormal basis $\mathcal{B}(\xi)$, there exists $f,g,h\in \mathcal{L}^2$, such that
\begin{equation}
    \widehat{M}(\xi)= f(\xi) \eta \otimes \eta 
    + h(\xi) \mu\otimes\mu
    +g(\xi) (\eta\otimes\mu+\mu\otimes\eta).
\end{equation}
However we know that $\tr(M)=0,$ and therefore
\begin{align}
    \tr\left(\widehat{M}\right)
    &=
    f-h \\
    &=0.
\end{align}
This means that
\begin{equation}
    \hat{M}(\xi)=f(\xi)(\eta\otimes\eta-\mu\otimes\mu)
    +g(\xi)(\eta\otimes\mu+\mu\otimes\eta),
\end{equation}
and so we may conclude that
\begin{equation}
    \left[\widehat{M}(\xi)\right]_{\mathcal{B}(\xi)}
    =
    \left[\begin{array}{ccc}
         0 & 0 &0  \\
         0 & f &g \\
         0& g & -f
    \end{array}\right],
\end{equation}
and
\begin{equation}
\left[ \widehat{M}\left( \xi \right)
\right]_{\mathcal{B}\left( \xi \right)}
\in \spn \left\{ \frac{1}{\sqrt{2}}\left[ 
\begin{array}{ccc}
0 & 0 & 0 \\ 
0 & 1 & 0 \\ 
0 & 0 & -1
\end{array}
\right] ,\frac{1}{\sqrt{2}}\left[ 
\begin{array}{ccc}
0 & 0 & 0 \\ 
0 & 0 & 1 \\ 
0 & 1 & 0
\end{array}
\right] \right\}.
\end{equation}
This completes the proof of the forward direction.

Now suppose $\widehat{M}\in \mathcal{L}^2$ and 
\begin{equation}
\left[ \widehat{M}\left( \xi \right)
\right]_{\mathcal{B}\left( \xi \right)}
\in \spn \left\{ \frac{1}{\sqrt{2}}\left[ 
\begin{array}{ccc}
0 & 0 & 0 \\ 
0 & 1 & 0 \\ 
0 & 0 & -1
\end{array}
\right] ,\frac{1}{\sqrt{2}}\left[ 
\begin{array}{ccc}
0 & 0 & 0 \\ 
0 & 0 & 1 \\ 
0 & 1 & 0
\end{array}
\right] \right\}.
\end{equation}
This implies that there exists $f,g\in L^2$ such that
\begin{equation}
    \hat{M}(\xi)=f(\xi)(\eta\otimes\eta-\mu\otimes\mu)
    +g(\xi)(\eta\otimes\mu+\mu\otimes\eta).
\end{equation}
Note that we can immediately see that for almost all $\xi \in \mathbb{R}^3,$
\begin{equation}
    \tr\left(\widehat{M}(\xi)\right)=0,
\end{equation}
and
\begin{equation} 
    \widehat{M}(\xi) \xi =0.
\end{equation}
Going from the Fourier space characterization back to the physical space, we can see that
\begin{equation}
    \tr(M)=0,
\end{equation}
and
\begin{equation}
    \divr(-\Delta)^{-\frac{1}{2}}M=0.
\end{equation}
This implies that $M\in\mathcal{L}^2_{tr\&divfree}$,
and this completes the proof.
\end{proof}

\begin{theorem} \label{HelmholtzTwo}
We have the further orthogonal decomposition
\begin{equation}
    \mathcal{L}^2_{divfree}
    =
    \mathcal{L}^2_{\widetilde{Id}} \oplus 
    \mathcal{L}^2_{tr\&divfree},
\end{equation}
and therefore the entire space $\mathcal{L}^2$ has the orthogonal decomposition
\begin{equation}
    \mathcal{L}^2=
    \mathcal{L}^2_{st} \oplus \mathcal{L}^2_{Hess} 
    \oplus L^2_{\widetilde{Id}} \oplus \mathcal{L}^2_{tr\&divfree}
\end{equation}
\end{theorem}

\begin{proof}
We have already shown that 
\begin{equation}
    \mathcal{L}^2=
    \mathcal{L}^2_{st} \oplus \mathcal{L}^2_{Hess} 
    \oplus \mathcal{L}^2_{divfree},
\end{equation}
so it suffices to show that
\begin{equation}
    \mathcal{L}^2_{divfree}
    =
    \mathcal{L}^2_{\widetilde{Id}} \oplus 
    \mathcal{L}^2_{tr\&divfree}.
\end{equation}
We will start by proving
\begin{equation}
    \mathcal{L}^2_{divfree}
    =
    \mathcal{L}^2_{\widetilde{Id}} 
    +\mathcal{L}^2_{tr\&divfree}.
\end{equation}
First we take $V_3$ as in Proposition \ref{FourierConstraint}, and define $W_1$ and $W_2$ as follows:
\begin{align}
    V_3&=
    \spn\left\{ \frac{1}{\sqrt{2}}\left( 
\begin{array}{ccc}
0 & 0 & 0 \\ 
0 & 1 & 0 \\ 
0 & 0 & 1
\end{array}
\right) ,\frac{1}{\sqrt{2}}\left( 
\begin{array}{ccc}
0 & 0 & 0 \\ 
0 & 1 & 0 \\ 
0 & 0 & -1
\end{array}
\right) ,\frac{1}{\sqrt{2}}\left( 
\begin{array}{ccc}
0 & 0 & 0 \\ 
0 & 0 & 1 \\ 
0 & 1 & 0
\end{array}
\right) \right\} \\
    W_1 &=
    \spn\left\{ \frac{1}{\sqrt{2}}\left( 
\begin{array}{ccc}
0 & 0 & 0 \\ 
0 & 1 & 0 \\ 
0 & 0 & 1
\end{array}
\right) \right\} \\
    W_2 &=
    \spn\left\{ 
    \frac{1}{\sqrt{2}}\left( 
\begin{array}{ccc}
0 & 0 & 0 \\ 
0 & 1 & 0 \\ 
0 & 0 & -1
\end{array}
\right),
\frac{1}{\sqrt{2}}\left( 
\begin{array}{ccc}
0 & 0 & 0 \\ 
0 & 0 & 1 \\ 
0 & 1 & 0
\end{array}
\right) \right\}.
\end{align}
It is immediately clear that
\begin{equation}
    V_3= W_1 \oplus W_2.
\end{equation}
From Proposition \ref{FourierConstraint}, we know that 
for all $M\in \mathcal{L}^2_{divfree},$ 
\begin{equation}
    \left[\widehat{M}(\xi)\right]_{\mathcal{B}(\xi)}
    \in V_3.
\end{equation}
Therefore, for each $\xi\in\mathbb{R}^3$ there exists a unique $\Tilde{M}^1(\xi)\in W_1, \Tilde{M}^2(\xi) \in W_2,$ such that
\begin{equation}
    \left[\widehat{M}(\xi)\right]_{\mathcal{B}(\xi)}
    =\Tilde{M}^1(\xi)+\Tilde{M^2}(\xi).
\end{equation}
Noting that clearly 
$\Tilde{M}^1,\Tilde{M}^2 \in\mathcal{L}^2$
and inverting the Fourier transform, take $M^1,M^2 \in\mathcal{L}^2$ such that
\begin{align}
    \left[\widehat{M^1}(\xi)\right]_{\mathcal{B}(\xi)}
    &= \Tilde{M}^1(\xi) \\
    \left[\widehat{M^2}(\xi)\right]_{\mathcal{B}(\xi)}
    &= \Tilde{M}^2(\xi).
\end{align}
Clearly we will then have
\begin{equation}
    M=M^1+M^2,
\end{equation}
and by Proposition \ref{FourierConstrainDivFree}, we have
$M^1\in \mathcal{L}^2_{\Tilde{Id}}$ and
$M^2 \in \mathcal{L}^2_{tr\&divfree}$.
We have now proven that
\begin{equation}
    \mathcal{L}^2_{divfree}
    =
    \mathcal{L}^2_{\widetilde{Id}} 
    +\mathcal{L}^2_{tr\&divfree}.
\end{equation}

It remains to prove orthogonality. Using the fact that the Fourier transform
is an isometry, and that the matrix inner product is independent of the
choice of orthonormal basis, we compute that for $M,S\in \mathcal{L}^{2}$ 
\begin{align}
\left\langle S,M\right\rangle & =\left\langle \widehat{S},\widehat{M}
\right\rangle \\
& =\int_{\mathbb{R}^{3}}
\left[ \widehat{M}(\xi )
\right]_{\mathcal{B}(\xi }
\odot \left[ \bar{\widehat{S}}(\xi )
\right]_{\mathcal{B}(\xi )} \diff\xi.
\end{align}
This implies that the orthogonality of $W_{1}$and $W_{2}$ together with
Proposition \ref{FourierConstrainDivFree}, which states the image of the Fourier
transform with respect to the basis $\mathcal{B}(\xi )$ 
must be in $W_1$ and $W_2$
for $\mathcal{L}_{\Tilde{Id}}^{2}$ and
$\mathcal{L}_{tr\&divfree}^{2}$ respectively, 
implies that $\mathcal{L}_{\Tilde{Id}}^{2}$ and
$\mathcal{L}_{tr\&divfree}^{2}$ are orthogonal. 
This completes the proof.
\end{proof}

\begin{remark}
As in the proof of Theorem \ref{HelmholtzOne}, the differential structure of the spaces in question is obscured in the proof of Theorem \ref{HelmholtzTwo}, because the differential structure is all buried in the choice of basis on the Fourier space side. The key fact is the orthogonality of $\mathcal{L}_{\widetilde{Id}}^{2}$ and $\mathcal{L}_{tr\&divfree}^{2}$.
We can see formally integrating by parts that
for all $f\in L^2, Q\in \mathcal{L}^2_{\tr\&divfree}$,
\begin{align}
    \left<fI_3+\Hess(-\Delta)^{-1}f,Q\right>
    &=
    \left<f,\tr(Q)\right>
    -\left<\nabla(-\Delta)^{-\frac{1}{2}}f,
    \divr(-\Delta)^{-\frac{1}{2}}Q\right> \\
    &=0.
\end{align}
The arguments here are precisely the same as on the Fourier side, but the choice of basis obscures what is really going on, whereas in the above computation the role of the divergence free and trace free constraints are clear.
\end{remark}

We will now justify the introduction of the space $\mathcal{L}^2_{\widetilde{Id}}$, as a space of adjusted identities, by showing that if is the space $\mathcal{L}^2_{Id}$ projected onto the orthogonal complement of $\mathcal{L}^2_{Hess}$.

\begin{proposition}
The space $\mathcal{L}^2_{Id}$ is the space of identity matrices projected to be orthogonal to the Hessians, with
\begin{equation}
    \mathcal{L}^2_{\widetilde{Id}}=
    P_{Hess}^\perp \mathcal{L}^2_{Id}.
\end{equation}
In particular, for all $f\in L^2,$ we have
\begin{equation} \label{AdjustedID}
    P_{Hess}^\perp \left(f I_3\right)=
    f I_3+
    \Hess(-\Delta)^{-1} f.
\end{equation}
\end{proposition}

\begin{proof}
It suffices to prove the identity \eqref{AdjustedID}.
We begin by observing that,
\begin{equation}
    \widehat{\left(fI_3\right)}=\hat{f}I_3.
\end{equation}
From the Fourier space decomposition in Theorem \ref{HelmholtzOne}, we can see that
\begin{equation}
    \widehat{P_{Hess}(fI_3)}=
    \hat{f} \frac{\xi}{|\xi|}\otimes \frac{\xi}{|\xi|},
\end{equation}
and therefore we can conclude that
\begin{equation}
    P_{Hess}(fI_3)=-\Hess(-\Delta)^{-1}f.
\end{equation}
It then follows that
\begin{equation}
    P_{Hess}^\perp(fI_3)= f I_3+ \Hess(-\Delta)^{-1} f,
\end{equation}
and this completes the proof.
\end{proof}

We can now give a complete representation of the orthogonal complement of the space of strain matrices, which was given in Corollary \ref{OrthogonalComplementIntro}, and is restated for the reader's convenience.

\begin{corollary} \label{OrthogonalComplement}
The orthogonal complement of the strain constraint space is given by
\begin{align}
    \left(\mathcal{L}^2_{st}\right)^\perp 
    &=
    \mathcal{L}^2_{Hess} \oplus
    \mathcal{L}^2_{\widetilde{Id}} \oplus
    \mathcal{L}^2_{tr\&divfree} \\
    &=
    \mathcal{L}^2_{Hess} 
\oplus \mathcal{L}^2_{divfree}.
\end{align}
\end{corollary}

\begin{proof}
This follows immediately from the orthogonal decompositions in Theorems \ref{HelmholtzOne} and \ref{HelmholtzTwo}.
\end{proof}

The decomposition in Theorem \ref{HelmholtzTwo} also allows us to give a distributional characterization of the strain space, analogous to the characterization that divergence free vector fields are those vector fields which are orthogonal to all gradients.
We will now prove Corollary \ref{Distribution}, which is restated for the reader's convenience.

\begin{corollary} \label{DistrubtionBody}
Suppose $S\in \mathcal{L}^2$. Then $S\in \mathcal{L}^2_{st}$
if and only if, for all $f,g\in C_c^\infty(\mathbb{R}^3)$
\begin{align}
    \left<S,\Hess(f)\right> &=0 \\
    \left<S,g I_3\right> &=0,
\end{align}
and for all $M\in C_c^\infty
\left(\mathbb{R}^3;\mathcal{S}^{3\times 3}\right)$, with
$\tr(M)=0$ and $\divr(M)=0$,
\begin{equation}
    \left<S,M\right>=0.
\end{equation}
\end{corollary}

\begin{proof}
The proof of this result essentially comes down to Corollary \ref{OrthogonalComplement} and the fact that $C_c^\infty$ is dense in $L^2$.
First we will observe that $S \in \mathcal{L}^2_{st}$ if and only if for all 
$M \in \left(\mathcal{L}^2_{st}\right)^\perp$,
\begin{equation}
    \left<S,M\right>=0.
\end{equation}
Suppose $S\in \mathcal{L}^2_{st}$. 
Then clearly for all $f,g\in L^2,$
\begin{align}
    \left<S,\Hess(-\Delta)^{-1}f\right>&=0 \\
    \left<S,\Hess(-\Delta)^{-1}g+gI_3\right>&=0,
\end{align}
and for all $M\in \mathcal{L}^2_{tr\&divfree}$,
\begin{equation}
    \left<S,M\right>=0.
\end{equation}
In particular this implies that for all $f\in L^2,$
\begin{align}
    \left<S, f I_3\right>
    &=
    \left<S,\Hess(-\Delta)^{-1}f+fI_3\right>
    -\left<S,\Hess(-\Delta)^{-1}f\right> \\
    &=
    0.
\end{align}
Using the fact that $C_c^\infty \subset \dot{H}^2$,
we can see that
for all $f,g\in C_c^\infty(\mathbb{R}^3)$
\begin{align}
    \left<S,\Hess(f)\right> &=0 \\
    \left<S,g I_3\right> &=0,
\end{align}
and for all $M\in C_c^\infty
\left(\mathbb{R}^3;\mathcal{S}^{3\times 3}\right)$, with
$\tr(M)=0$ and $\divr(M)=0$,
\begin{equation}
    \left<S,M\right>=0.
\end{equation}

Now we will consider the reverse direction. 
Suppose for all $f,g\in C_c^\infty(\mathbb{R}^3)$
\begin{align}
    \left<S,\Hess(f)\right> &=0 \\
    \left<S,g I_3\right> &=0,
\end{align}
and for all $M\in C_c^\infty
\left(\mathbb{R}^3;\mathcal{S}^{3\times 3}\right)$, with
$\tr(M)=0$ and $\divr(M)=0$,
\begin{equation}
    \left<S,M\right>=0.
\end{equation}
Using the fact that $C_c^\infty$ is dense in $\dot{H}^2$, we can conclude that for all $g\in \dot{H}^2$,
\begin{equation}
    \left<S,\Hess(g)\right> =0,
\end{equation}
and consequently for all $f\in L^2$,
\begin{equation} \label{PerpHess}
    \left<S,\Hess(-\Delta)^{-1}f\right> =0.
\end{equation}
Likewise, using the fact that $C_c^\infty$ is dense in $L^2,$
we can conclude for all $g\in L^2$,
\begin{equation} \label{PerpID}
    \left<S,g I_3\right> =0;
\end{equation}
and we can also conclude for all $M\in \mathcal{L}^2_{tr\&divfree}$,
\begin{equation}
    \left<S,M\right>=0.
\end{equation}
Putting together \eqref{PerpHess} and \eqref{PerpID}, we can conclude that 
for all $f\in L^2$,
\begin{equation}
    \left<S,\Hess(-\Delta)^{-1}f+f I_3\right>=0.
\end{equation}
Therefore, applying Corollary \ref{OrthogonalComplement}, which states
\begin{equation}
\left(\mathcal{L}^2_{st}\right)^\perp 
    =
    \mathcal{L}^2_{Hess} \oplus
    \mathcal{L}^2_{\widetilde{Id}} \oplus
    \mathcal{L}^2_{tr\&divfree},
\end{equation}
we find that for all 
$M\in \left(\mathcal{L}^2_{st}\right)^\perp$,
\begin{equation}
    \left<S,M\right>=0.
\end{equation}
This implies that $S\in\mathcal{L}^2_{st}$
and completes the proof.
\end{proof}

The orthogonal complement of Hessians is even more straightforward, as it is simply the kernel of divergence squared. From this it follows that the space of strain matrices can be expressed as the direct difference of the kernel of divergence squared and the kernel of divergence.

\begin{proposition} \label{HessOrthog}
The orthogonal complement of the space of Hessians is given by the set of matrices that are divergence squared free,
\begin{equation}
    \left(\mathcal{L}^2_{Hess}\right)^\perp
    = \ker\left(\divr^2\right)
\end{equation}
\end{proposition}

\begin{proof}
Suppose $M\in\mathcal{L}^2, \divr^2(-\Delta)^{-1}M=0$.
Then integrating by parts twice we can see that 
for all $f\in L^2$,
\begin{align}
    \left<\Hess(-\Delta)^{-1}f,M\right>
    &=
    \left<f, \divr^2(-\Delta)^{-1}M\right> \\
    &=0,
\end{align}
and therefore $M\in\left(\mathcal{L}^2_{Hess}\right)^\perp$

Now suppose $M\in\left(\mathcal{L}^2_{Hess}\right)^\perp$.
Again integrating by parts twice, this implies that for all 
$f\in L^2,$
\begin{align}
    \left<\divr^2(-\Delta)^{-1}M,f\right>
    &=
    \left<M,\Hess(-\Delta)^{-1}f\right> \\
    &=0.
\end{align}
Using the weak side characterization, this implies that 
$\divr^2(-\Delta)^{-1}M=0$ in $L^2$,
 and this completes the proof.
\end{proof}

We will now prove Theorem \ref{DifferenceThm}, which is restated for the reader's convenience.
\begin{theorem}
\begin{equation} \label{DifferenceA}
  \mathcal{L}^2_{st}=\ker\left(\divr^2\right) \ominus \ker(\divr),
\end{equation}
and equivalently
\begin{equation} \label{DifferenceB}
    \ker\left(\divr^2\right)= 
    \ker(\divr) \oplus \mathcal{L}^2_{st}.
\end{equation}
\end{theorem}

\begin{proof}
First we will define the direct difference
for three Hilbert spaces, $X,Y,Z$.
We will say
$Z=X\ominus Y$ if and only if
$X=Y \oplus Z$.
The statements \eqref{DifferenceA} and \eqref{DifferenceB} are equivalent from the definition of the direct difference, so it suffices to prove \eqref{DifferenceB}.
We know from Theorem \ref{HelmholtzOne} that
\begin{equation}
    \mathcal{L}^2=\mathcal{L}^2_{Hess}
    \oplus \mathcal{L}^2_{st} \oplus \ker(\divr),
\end{equation}
and so it is clear that 
\begin{equation}
    \left(\mathcal{L}^2_{Hess}\right)^\perp
    =\mathcal{L}^2_{st} \oplus \ker(\divr).
\end{equation}
However, we showed in Proposition \ref{HessOrthog}, that
\begin{equation}
    \left(\mathcal{L}^2_{Hess}\right)^\perp
    = \ker \left(\divr^2\right),
\end{equation}
and so we may conclude that
\begin{equation} 
    \ker\left(\divr^2\right)= 
    \ker(\divr) \oplus \mathcal{L}^2_{st}.
\end{equation}
This completes the proof.
\end{proof}

Related to this characterization, we can compute the size of the strain projection $P_{st}$ of a matrix in terms of the curl of its divergence.
We now prove the identity \eqref{CurlFormula}, which is restated for the reader's convenience.

\begin{proposition}
For all $M\in\mathcal{L}^2,$ we have the identity for the magnitude of the projection of $M$ onto $\mathcal{L}^2_{st}$
\begin{equation}
    \left\|P_{st}(M)\right\|_{L^2}^2=
    2\left\|\nabla\times\divr(-\Delta)^{-1}
    M\right\|_{L^2}^2.
\end{equation}
\end{proposition}

\begin{proof}
We will begin by letting
\begin{equation}
    M=\Hess(-\Delta)^{-1}f
    +\nabla_{sym}(-\Delta)^{-\frac{1}{2}}u
    +Q,
\end{equation}
where $\nabla\cdot u=0, \divr(Q)=0$.
Taking the divergence of this representation, we find
\begin{equation}
    -2\divr(M)=2\nabla f+(-\Delta)^{\frac{1}{2}}u,
\end{equation}
and taking the curl and inverting the Laplacian we find that
\begin{equation}
    -2\nabla\times\divr(-\Delta)^{-1}M= 
    \nabla\times (-\Delta)^{-\frac{1}{2}}u.
\end{equation}
Applying Proposition \ref{isometry}, we find that
\begin{align}
    \left\|P_{st}(M)\right\|_{L^2}^2
    &=
    \left\|\nabla_{sym}(-\Delta)^{-\frac{1}{2}}
    u\right\|_{L^2}^2 \\
    &=
    \frac{1}{2}\left\|\nabla\times
    (-\Delta)^{-\frac{1}{2}}u\right\|_{L^2}^2 \\
    &=
    2\left\|\nabla\times\divr(-\Delta)^{-1}
    M\right\|_{L^2}^2.
\end{align}
\end{proof}

We will now consider the the rotational invariants of the space $\mathcal{L}^2_{st}$ proving Proposition \ref{RotInvarIntro}, which is restated here for the reader's convenience.

\begin{proposition} \label{RotInvar}
For all $S\in\mathcal{L}^2_{st}$, and for all
rotation matrices $Q\in SO(3),$ $S^Q \in\mathcal{L}^2_{st}$,
where
\begin{equation}
    S^Q(x)= Q^{tr} S(Qx) Q
\end{equation}
\end{proposition}

\begin{proof} 
Suppose $S\in\mathcal{L}^2_{st}$ and $Q\in SO(3)$.
We know that there exists a unique 
$u\in\mathbf{\dot{H}}^1_{df},$ such that
\begin{equation}
    S=\nabla_{sym}u.
\end{equation}
Following the standard approach for the rotational invariance for the space of divergence free vector fields---see for instance Chapter 1 in \cite{MajdaBertozzi}---let
\begin{equation}
    u^Q(x)=Q^{tr}u(Qx)
\end{equation}
Observe that
\begin{align}
    u^Q_j(x)
    &=
    \sum_{k=1}^3 Q^{tr}_{jk} u_k(Qx) \\
    &=
    \sum_{k=1}^3 Q_{kj} u_k(Qx)
\end{align}
Applying the chain rule we find that
\begin{equation}
    \partial_i u^Q_j (x)= 
    \sum_{k,m=1}^3
    Q_{kj} Q_{mi} \left(\partial_m u_k\right)(Qx).
\end{equation}
Writing out this gradient as a matrix, we can see that
\begin{equation}
    \nabla u^Q(x)=Q^{tr} \nabla u(Qx) Q.
\end{equation}

Because $Q\in SO(3),$ we can immediately see that 
\begin{equation}
    \|\nabla u^Q\|_{L^2}=\|\nabla u\|_{L^2},
\end{equation}
and furthermore that 
\begin{align}
    \nabla \cdot u^Q(x)
    &=
    \tr\left(\nabla u^Q(x)\right) \\
    &=
    \tr\left(Q^{tr}\nabla u(Qx) Q\right) \\
    &=
    \tr \left(\nabla u(Qx)\right) \\
    &=
    (\nabla \cdot u)(Qx)\\
    &=
    0.
\end{align}
Therefore we can see that $u^Q \in \mathbf{H}^1_{df}$.
Finally we observe that
\begin{align}
    \nabla_{sym} u^Q(x)
    &=
    \frac{1}{2}\nabla u^Q(x)
    +\frac{1}{2}\left(\nabla u^Q\right)^{tr}(x) \\
    &=
    \frac{1}{2}Q^{tr}\nabla u(Qx) Q
    +\frac{1}{2}Q^{tr} \left(\nabla u\right)^{tr}(Qx)Q \\
    &=
    Q^{tr} S(Qx) Q.
\end{align}
Therefore we can see that $S^Q=\nabla_{sym} u^Q$, and so $S^Q\in\mathcal{L}^2_{st}$,
and this completes the proof.
\end{proof}

\section{Inequalities and projection bounds}
\label{InequalitiesSection}

In this section, we consider the projection of max-mid matrices onto the space $\mathcal{L}^2_{st}$, in particular proving Theorem \ref{OneDirectionMaxIntro}, which will be restated in this section as Theorem \ref{OneDirectionMax}.
We will begin by proving a sharp inequality for a single diagonal entry of a matrix $S\in \mathcal{L}^2_{st}$ in terms of a general, fixed basis.

\begin{theorem} \label{OneDiagCompBound}
For all $S\in\mathcal{L}^2_{st}$
\begin{equation}
    \|S_{33}\|_{L^2}^2 \leq 
    \frac{1}{2} \|S\|_{L^2}^2
\end{equation}
More generally, for all $S\in\mathcal{L}^2_{st}$ and
for all fixed unit vectors $v\in\mathbb{R}^3, |v|=1,$
\begin{equation} \label{OneDirectionBound}
    \|S \odot (v\otimes v)\|_{L^2}^2 \leq 
    \frac{1}{2} \|S\|_{L^2}^2
\end{equation}
\end{theorem}

\begin{proof}
Noting that $S_{33}=S\odot (e_3\otimes e_3)$, it suffices to prove the inequality \eqref{OneDirectionBound}.
We know that $S\in \mathcal{L}^2_{st},$ so let
\begin{equation}
    S=\nabla_{sym}(-\Delta)^{-\frac{1}{2}}u,
\end{equation}
for some $u\in \mathbf{L}^2_{df}$.
Applying the isometry from Proposition \ref{isometry}, we can conclude that
\begin{equation}
    \|S\|_{L^2}^2=\frac{1}{2}\|u\|_{L^2}^2.
\end{equation}
Using the fact that the Fourier transform is an isometry on $L^2$ and the fact that $v\in\mathbb{R}^3$ is a fixed vector we compute that
\begin{align}
    \|S \odot (v\otimes v)\|_{L^2}^2
    &=
    \left\|\mathcal{F}\left(S \odot 
    (v\otimes v)\right)\right\|_{L^2}^2 \\
    &=
    \left\|\hat{S} \odot(v\otimes v)\right\|_{L^2}^2.
\end{align}
Observing that 
\begin{equation}
    \hat{S}=\frac{i}{2}\left(
    \frac{\xi}{|\xi|}\otimes \hat{u}
    +\hat{u}\otimes \frac{\xi}{|\xi|}
    \right),
\end{equation}
we can compute that
\begin{equation}
    \hat{S}\odot (v\otimes v)=
    \left(v\cdot \frac{\xi}{|\xi|}\right)
    \left(v\cdot \hat{u}\right),
\end{equation}
and so
\begin{align}
    \|S \odot (v\otimes v)\|_{L^2}^2
    &=
    \int_{\mathbb{R}^3} 
    \left(v\cdot \frac{\xi}{|\xi|}\right)^2
    \left(v\cdot \hat{u}(\xi)\right)^2 \diff\xi \\
    &=
    \int_{\mathbb{R}^3} 
    \left(v\cdot \frac{\xi}{|\xi|}\right)^2
    \left(v\cdot \frac{\hat{u}(\xi)}
    {\left|\hat{u}(\xi)\right|}\right)^2
    \left|\hat{u}(\xi)\right|^2 \diff\xi.
\end{align}

Using the divergence free constraint, we can see that 
$\frac{\hat{u}(\xi)}{\left|\hat{u}(\xi)\right|}$ and
$\frac{\xi}{|\xi|}$ are orthonormal vectors, and therefore
\begin{align}
    \left(v\cdot \frac{\xi}{|\xi|}\right)^2
    +\left(v\cdot \frac{\hat{u}(\xi)}
    {\left|\hat{u}(\xi)\right|}\right)^2
    &\leq
    |v|^2 \\
    &=
    1,
\end{align}
so we may conclude that
\begin{equation}
    \left(v\cdot \frac{\hat{u}(\xi)}
    {\left|\hat{u}(\xi)\right|}\right)^2
    \leq
    1- \left(v\cdot \frac{\xi}{|\xi|}\right)^2.
\end{equation}
Setting $r=\left(v\cdot \frac{\xi}{|\xi|}\right)^2,$ observing that $0\leq r \leq 1,$ and applying Plancherel's identity and the Proposition \ref{isometry}, we find that
\begin{align}
    \|S \odot (v\otimes v)\|_{L^2}^2
    &\leq 
    \int_{\mathbb{R}^3}
    \left(v\cdot \frac{\xi}{|\xi|}\right)^2
    \left(1- \left(v\cdot \frac{\xi}{|\xi|}\right)^2\right)
    \left|\hat{u}(\xi)\right|^2 \diff\xi \\
    &\leq
    \int_{\mathbb{R}^3}
    \left(\sup_{0\leq r\leq 1} r(1-r)\right)
    \left|\hat{u}(\xi)\right|^2 \diff\xi \\
    &=
    \frac{1}{4}\int_{\mathbb{R}^3}
    \left|\hat{u}(\xi)\right|^2 \diff\xi \\
    &=
    \frac{1}{4}\|u\|_{L^2}^2 \\
    &=
    \frac{1}{2}\|S\|_{L^2}^2.
\end{align}
This completes the proof.

\end{proof}

We will now show that this inequality is sharp by providing a family of near maximizers.

\begin{theorem} \label{OneDiagCompSharp}
Suppose $S\in\mathcal{L}^2_{st}$ and 
$\supp\left(\hat{S}\right) \subset 
\left\{\xi\in\mathbb{R}^3: 
\frac{1}{2}-\frac{\epsilon}{2}<\frac{\xi_3^2}{|\xi|^2}
<\frac{1}{2}+\frac{\epsilon}{2}\right\},$ 
for some $0<\epsilon<1$.
Furthermore, suppose $S=\nabla_{sym}u$ with $u\in \dot{H}^1_{df}$ satisfying
$e_\theta \cdot \hat{u}(\xi)=0,$ for all $\xi\in\mathbb{R}^3.$
Then
\begin{equation}
    \|S_{33}\|_{L^2}^2 
    \geq \frac{1}{2}(1-\epsilon)^2 \|S\|_{L^2}^2.
\end{equation}
\end{theorem}

\begin{proof}
Recall that the basis for $\mathbb{R}^3$ in cylindrical coordinates is given by
\begin{align}
    e_r &= 
    \frac{1}{\left(\xi_1^2+\xi_2^2\right)^\frac{1}{2}}
    \left(\begin{array}{ccc}
         \xi_1  \\
         \xi_2 \\
         0
    \end{array}\right) \\
        e_\theta &= 
    \frac{1}{\left(\xi_1^2+\xi_2^2\right)^\frac{1}{2}}
    \left(\begin{array}{ccc}
         -\xi_2  \\
         \xi_1 \\
         0
    \end{array}\right) \\
    e_3 &=
    \left(\begin{array}{ccc}
        0  \\
        0 \\
        1
    \end{array}\right).
\end{align}
We know that $e_\theta\cdot \hat{u}(\xi)=0$, so
\begin{equation}
    \hat{u}(\xi)= f(\xi) e_r+ g(\xi) e_3.
\end{equation}
Observe that $\xi=\left(\xi_1^2+\xi_2^2\right)^\frac{1}{2}
e_r+\xi_3 e_3$, and using the constraint $\nabla\cdot u=0$, 
we find
\begin{align}
    \xi\cdot\hat{u}(\xi)
    &=
    \left(\xi_1^2+\xi_2^2\right)^\frac{1}{2} f(\xi)
    +\xi_3 g(\xi) \\
    &=0.
\end{align}
Therefore 
\begin{equation}
    f(\xi)=-\frac{\xi_3}
    {\left(\xi_2^1+\xi_2^2\right)^\frac{1}{2}} g(\xi),
\end{equation}
and so
\begin{equation}
    \hat{u}(\xi)=g(\xi) \left( e_3
    -\frac{\xi_3}{\left(\xi_1^2+\xi_2^2\right)^\frac{1}{2}}
    e_r \right).
\end{equation}
In particular this means that
\begin{align}
    \left|\hat{u}(\xi)\right|^2
    &=
    |g(\xi)|^2\left(1+\frac{\xi_3^2}{\xi_1^2+\xi_2^2}\right) \\
    &=
    |g(\xi)|^2 \left(\frac{\xi_1^2+\xi_2^2+\xi_3^2}
    {\xi_1^2+\xi_2^2}\right)\\
    &=
    |g(\xi)|^2 \left(\frac{|\xi|^2}{|\xi|^2-\xi_3^2}\right).
\end{align}

Next we will note that 
\begin{align}
    \|S_{33}\|_{L^2}^2
    &=
    \|\partial_3 u_3\|_{L^2}^2 \\
    &=
    \left\|\widehat{\partial_3 u_3}\right\|_{L^2}^2 \\
    &=
    \int_{\mathbb{R}^3} 4\pi^2\xi_3^2
    \left|g(\xi)\right|^2 \diff\xi \\
    &=
    \int_{\mathbb{R}^3} 4\pi^2\xi_3^2
    \left|\hat{u}(\xi)\right|^2
    \left(\frac{|\xi|^2-\xi_3^2}{|\xi|^2}\right)\diff\xi \\
    &=
    \int_{\mathbb{R}^3} 4\pi^2|\xi|^2
    \left|\hat{u}(\xi)\right|^2
    \frac{\xi_3^2}{|\xi|^2}
    \left(1-\frac{\xi_3^2}{|\xi|^2}\right)\diff\xi .
\end{align}
Recalling that by hypothesis, for all $\xi\in\mathbb{R}^3$ such that $\hat{u}(\xi)\neq 0$,
\begin{equation}
    \frac{1}{2}-\frac{\epsilon}{2}
    <\frac{\xi_3^2}{|\xi|^2}
    <\frac{1}{2}+\frac{\epsilon}{2},
\end{equation}
we can see that
\begin{equation}
    \frac{\xi_3^2}{|\xi|^2}
    \left(1-\frac{\xi_3^2}{|\xi|^2}\right)
    >
    \left(\frac{1}{2}-\frac{\epsilon}{2}\right)^2.
\end{equation}
Applying Plancherel's formula and Proposition \ref{isometry}, we find that
\begin{align}
    \|S_{33}\|_{L^2}^2
    &\geq
    \left(\frac{1}{2}-\frac{\epsilon}{2}\right)^2
    \int_{\mathbb{R}^3} 4 \pi^2 |\xi|^2 
    \left|\hat{u}(\xi)\right|^2 \diff\xi \\
    &=
    \left(\frac{1}{2}-\frac{\epsilon}{2}\right)^2
    \|\nabla u\|_{L^2}^2 \\
    &=
    \left(\frac{1}{2}-\frac{\epsilon}{2}\right)^2
    2\|S\|_{L^2}^2 \\
    &=
    (1-\epsilon)^2 \frac{1}{2}\|S\|_{L^2}^2
\end{align}
\end{proof}

\begin{theorem} \label{OneDirectionMax}
Let $v\in\mathbb{R}^3$ be a fixed unit vector, $|v|=1$, then purely axial compression along the $v$ axis, and planar stretching in the plane perpendicular to $v$ is not possible. That is there does not exist 
$\lambda \in L^2, \lambda\geq 0$,
not identically zero, such that
\begin{equation}
    \frac{\lambda}{\sqrt{6}}\left(I_3-v\otimes v\right)
    \in \mathcal{L}^2_{st}.
\end{equation}
In fact, for all $v\in\mathbb{R}^3, |v|=1$ we have the explicit bound on the projection of matrices of the form onto $\mathcal{L}^2_{st}$,
\begin{equation}
    \sup_{\substack{\|\lambda\|_{L^2}=1 \\
    \lambda \geq 0}}
    \left\|P_{st}\left(\frac{\lambda}{\sqrt{6}}
    \left(I_3-3v \otimes v\right)\right)
    \right\|_{L^2}^2=\frac{3}{4}.
\end{equation}

We can also give explicitly a family of near maximizers for the case $v=e_3$ as follows. Suppose
$\|\lambda\|_{L^2}=1$ and
$\supp\left(\hat{\lambda}\right) \in 
\left\{\xi\in\mathbb{R}^3: 
\frac{1}{2}-\frac{\epsilon}{2}
<\frac{\xi_3^2}{|\xi|^2}
<\frac{1}{2}+\frac{\epsilon}{2}\right\}$.
Then
\begin{equation}
    \left\|P_{st}\left(\frac{\lambda}{\sqrt{6}}
    \left(I_3-3 e_3\otimes e_3\right)\right)\right\|_{L^2}
    >\frac{3}{4}(1-\epsilon)^2.
\end{equation}
The analogous family of maximizers for any other fixed $v\in\mathbb{R}^3, |v|=1$ can be immediately deduced by the rotational symmetry of $\mathcal{L}^2_{st}$,
Proposition \ref{RotInvar}.
\end{theorem}

\begin{proof}
First we will prove that for all $\lambda\in L^2,
\|\lambda\|_{L^2}=1,$ 
\begin{equation}
    \left\|P_{st}\left(\frac{\lambda}{\sqrt{6}}
    \left(I_3-3v \otimes v\right)\right)
    \right\|_{L^2}^2\leq \frac{3}{4}.
\end{equation}
Applying Corollary \ref{DistrubtionBody} and Theorem \ref{OneDiagCompBound},
we can see that for all $S\in \mathcal{L}^2_{st}, 
\|S\|_{L^2}=1$
\begin{align}
    \left< \frac{\lambda}{\sqrt{6}}
    \left(I_3-3v \otimes v\right); S\right>
    &=
    -\frac{3}{\sqrt{6}}
    \left< \lambda v \otimes v; S\right> \\
    &=
    -\frac{\sqrt{3}}{\sqrt{2}}
    \left< \lambda; S\odot (v \otimes v)\right> \\
    &\leq
    \frac{\sqrt{3}}{\sqrt{2}}\|\lambda\|_{L^2}
    \|S\cdot (v\otimes v)\|_{L^2} \\
    &\leq
    \frac{\sqrt{3}}{2}\|S\|_{L^2} \\
    &=
    \frac{3}{\sqrt{2}}.
\end{align}
Using the dual space characterization of the $L^2$ norm,
we can see that
\begin{align}
    \left\|P_{st}\left(\frac{\lambda}{\sqrt{6}}
    \left(I_3-3v \otimes v\right)\right)
    \right\|_{L^2}
    &=
    \sup_{\substack{S\in \mathcal{L}^2_{st} \\
    \|S\|_{L^2}=1}} \left<\left(\frac{\lambda}
    {\sqrt{6}}\left(I_3-3v \otimes v\right)\right);S\right> \\
    &\leq
    \frac{\sqrt{3}}{2}.
\end{align}
Therefore, we can clearly see that for all $\|\lambda\|_{L^2}=1,$ $v\in\mathbb{R}^3$ fixed, $|v|=1$,
\begin{equation}
    \left\|P_{st}\left(\frac{\lambda}{\sqrt{6}}
    \left(I_3-3v \otimes v\right)\right)
    \right\|_{L^2}^2\leq \frac{3}{4},
\end{equation}
which implies that
\begin{equation}
    \sup_{\substack{\|\lambda\|_{L^2}=1 \\
    \lambda \geq 0}}
    \left\|P_{st}\left(\frac{\lambda}{\sqrt{6}}
    \left(I_3-3v \otimes v\right)\right)
    \right\|_{L^2}^2 \leq \frac{3}{4}.
\end{equation}

Now we will show this inequality holds with equality for all $v\in\mathbb{R}^3, |v|=1$.
Using the rotational invariance of the space $\mathcal{L}^2_{st}$ from Proposition \ref{RotInvarIntro}, we can take $v=e_3$ without loss of generality.
Fix $0<\epsilon<1,$ and fix  $\lambda\in L^2$ such that $\lambda\geq 0$, $\|\lambda\|_{L^2}=1$, and
$\supp\left(\hat{\lambda}\right) \in 
\left\{\xi\in\mathbb{R}^3:
\frac{1}{2}-\frac{\epsilon}{2}
<\frac{\xi_3^2}{|\xi|^2}
<\frac{1}{2}+\frac{\epsilon}{2}
\right\}$.
Let 
\begin{equation}
\hat{S}(\xi)=\frac{1}{\sqrt{2}} \hat{\lambda}(\xi)
\sgn\left(\frac{\xi_3}{|\xi|}\right)
\left(\frac{\xi}{|\xi|}\otimes w(\xi)
+w(\xi) \otimes \frac{\xi}{|\xi|}\right),    
\end{equation}
where
\begin{equation}
    w(\xi)=\frac{1}{|\xi|\left(|\xi|^2-\xi_3^2
    \right)^\frac{1}{2}}\left(\xi_3\xi-|\xi|^2e_3\right).
\end{equation}
Note that 
\begin{equation}
    \xi \cdot w(\xi)=0,
\end{equation}
so by Proposition \ref{FourierConstraint}, $S\in\mathcal{L}^2$.
Furthermore, note that 
\begin{align}
    |w(\xi)|^2  &=
    \frac{1}{|\xi|^2\left(|\xi|^2-\xi_3^2\right)}
    (|\xi|^4- \xi_3^2|\xi|^2) \\
    &= 1,
\end{align}
and so we can conclude that $\|S\|_{L^2}=1$.
Next we will take the inner product, using Plancherel's formula to find that
\begin{align}
        \left<\frac{\lambda}{\sqrt{6}}
        \left(I_3- 3 e_3 \otimes e_3\right),S\right>
        &=
        -\frac{\sqrt{3}}{\sqrt{2}}
        \left<\lambda,S_{33}\right> \\
        &=
        -\frac{\sqrt{3}}{\sqrt{2}}
        \left<\hat{\lambda},\hat{S}_{33}\right> \\
        &=
        \sqrt{3} \int_{\mathbb{R}^3} 
        |\hat{\lambda}(\xi)|^2
        \sgn\left(\frac{\xi_3}{|\xi|}\right)
        \frac{\xi_3}{|\xi|}
        \frac{\sqrt{|\xi|^2-\xi_3^2}}{|\xi|}
        \diff\xi. \\
        &=
        \sqrt{3} \int_{\mathbb{R}^3} 
        |\hat{\lambda}(\xi)|^2
        \left(
        \frac{\xi_3^2}{|\xi|^2}
        \frac{|\xi|^2-\xi_3^2}{|\xi|^2}
        \right)^\frac{1}{2} \diff\xi .
\end{align}
Applying the condition that 
for all $\xi\in\supp\left(\hat{\lambda}\right),$
\begin{equation}
    \frac{1}{2}-\frac{\epsilon}{2}
    <\frac{\xi_3^2}{|\xi|^2}
    <\frac{1}{2}+\frac{\epsilon}{2},
\end{equation}
we find that
\begin{align}
    \left<\frac{\lambda}{\sqrt{6}}
    \left(I_3- 3 e_3 \otimes e_3\right),S\right>
    &=
    \sqrt{3} \int_{\mathbb{R}^3} 
    |\hat{\lambda}(\xi)|^2
    \left(
    \frac{\xi_3^2}{|\xi|^2}
    \left(1-\frac{\xi_3^2}{|\xi|^2}\right)
    \right)^\frac{1}{2} \diff\xi \\
    &>
    \sqrt{3}\left(\frac{1}{2}-\frac{\epsilon}{2}\right) \|\lambda\|_{L^2}^2 \\
    &=
    \frac{\sqrt{3}}{2}(1-\epsilon).
\end{align}
Recalling that $S\in\mathcal{L}^2_{st}, \|S\|_{L^2}=1,$
we find that 
\begin{equation}
    \left\|P_{st}\left(\frac{\lambda}{\sqrt{6}}
    \left(I_3-3 e_3\otimes e_3\right)\right)
    \right\|_{L^2}
    >\frac{\sqrt{3}}{2}(1-\epsilon).
\end{equation}
Therefore we can conclude that for all $0<\epsilon<1,$ 
there exists 
$\lambda\in L^2, \lambda\geq 0, \|\lambda\|_{L^2}=1$
such that
\begin{equation}
    \left\|P_{st}\left(\frac{\lambda}{\sqrt{6}}
    \left(I_3-3 e_3\otimes e_3\right)\right)\right\|_{L^2}
    >\frac{3}{4}(1-\epsilon)^2.
\end{equation}

While we have shown that the family of near maximizers gets arbitrarily close to $\frac{3}{4}$, 
the stipulation that $\lambda\geq 0$ is not addressed by that condition. We will now give an example of a sequence satisfying the condition that $\lambda\geq 0$ that comes arbitrarily close to $\frac{3}{4}$.
Let
\begin{align}
    v_1&=
    \left(\begin{array}{c}
         1\\
         0 \\
         0 
    \end{array}\right) \\
    v_2&=\frac{1}{\sqrt{2}}
    \left(\begin{array}{c}
         0\\
         -1 \\
         1 
    \end{array}\right) \\
     v_3&=\frac{1}{\sqrt{2}}
    \left(\begin{array}{c}
         0\\
         1 \\
         1 
    \end{array}\right).
\end{align}
Now we define $\lambda$ by letting its Fourier transform be given by
\begin{equation}
    \hat{\lambda}_n(\xi)=a_n
    \exp\left(-n(v_1\cdot \xi)^2-n(v_2\cdot\xi)^2
    -\frac{1}{n}(v_3\cdot\xi)^2\right),
\end{equation}
where $a_n$ is chosen such that
\begin{equation}
    \|\lambda_n\|_{L^2}=
    \|\hat{\lambda}_n\|_{L^2}
    =1.
\end{equation}
We can see that $\hat{\lambda}_n$ is a Gaussian function, and the Fourier transform preserves the set of Gaussian functions, and so we can conclude that $\lambda_n$ is Gaussian and that for all $x \in\mathbb{R}^3, \lambda_n(x)>0.$
Finally, we can see that in the limit 
$n\to\infty$ the support of $\hat{\lambda}_n$ concentrates on the set $\spn\left\{v_3\right\}$. Noting that
\begin{equation}
    \spn\left\{v_3\right\}\subset
    \left\{\xi\in\mathbb{R}^3: 
    \frac{\xi_3^2}{|\xi|^2}
    =\frac{1}{2}\right\},
\end{equation}
we can conclude that
\begin{equation}
    \lim_{n\to\infty}
    \left\|P_{st}\left(\frac{\lambda}{\sqrt{6}}
    \left(I_3-3e_3\otimes e_3\right)
    \right)\right\|_{L^2}^2
    =\frac{3}{4},
\end{equation}
and this completes the proof.
\end{proof}

We will now show that any nontrivial matrix in $\mathcal{L}^2_{maxmid}\cap \mathcal{L}^2_{st}$ must stretch in all three directions, proving Theorem \ref{TwoDirectionCaseIntro}.

\begin{theorem} \label{TwoDirectionCase}
Suppose that $|v(x)|=1$ and $v_1(x)=0$ almost everywhere $x\in\mathbb{R}^3$. That is, suppose $v$ is a unit vector varying in space, but restricted to the $yz$ plane.
Then there does not exist $\lambda\in L^2, \lambda\geq 0,$ and $\lambda$ not identically zero, such that
\begin{equation}
    \lambda(I_3-3v\otimes v)\in \mathcal{L}^2_{st}.
\end{equation}
This means that any nontrivial maxmid matrix in the strain constraint space must compress in all three directions. The axis of compression cannot be confined to any fixed plane.
\end{theorem}

\begin{proof}
Suppose towards contradiction that there exists
$\lambda\in L^2, \lambda\geq 0,$ and $\lambda$ not identically zero,
and $v\in \mathbf{L}^\infty$ with $|v(x)|=1$ and $v_1(x)=0$ almost everywhere $x\in\mathbb{R}^3$,
such that
\begin{equation}
    \lambda(I_3-3v\otimes v)\in \mathcal{L}^2_{st}.
\end{equation}
Then there exists $u\in\mathbf{\dot{H}}^1_{df}$, such that
\begin{equation}
     \lambda(I_3-3v\otimes v)=\nabla_{sym}u.
\end{equation}
Taking the $(1,1)$ entry of each of these matrices, we find that
\begin{equation}
    \partial_1 u_1=\lambda.
\end{equation}
We know by hypothesis that $\lambda$ is not identically zero---that is $\lambda>0$ on a set of positive measure---and so there must exist a set of positive measure $\Omega\subset \mathbb{R}^2$, such that
\begin{equation}
    \int_{-\infty}^{+\infty} \lambda(\tau,x_2,x_3)
    \diff\tau>0,
\end{equation}
for all $(x_2,x_3)\in\Omega$.
Using the fact that $\partial_1u_1=\lambda$,
we can conclude that for all
$(x_2,x_3)\in\Omega$,
\begin{align}
    \lim_{b\to+\infty} u_1(b,x_2,x_3)
    -\lim_{a\to-\infty}u_1(a,x_2,x_3)
    &=\int_{-\infty}^{+\infty}
    \lambda(\tau,x_2,x_3) \diff\tau \\
    &>0.
\end{align}
This implies that either there exists a set of positive measure, $\Omega^+$, such that for all $(x_2,x_3)\in\Omega^+$,
\begin{equation}
    \lim_{b\to+\infty} u_1(b,x_2,x_3)>0,
\end{equation}
or there exists a set of positive measure, $\Omega^-$, such that for all $(x_2,x_3)\in\Omega^-$,
\begin{equation}
    \lim_{a\to -\infty} u_1(a,x_2,x_3)<0.
\end{equation}
Both of these possibilities contradict the assumption that
$u\in \mathbf{\dot{H}}^1 \subset\mathbf{L}^6$,
and this completes the proof.
\end{proof}

We will now prove Theorem \ref{EigenGapIntro}, showing that a gap in the projection of maxmid matrices onto the space $\mathcal{L}^2_{st}$ implies a gap between the largest and intermediate eigenvalues of any strain matrix.

\begin{theorem} \label{EigenGap}
Suppose
\begin{equation} \label{SupHypothesis}
    \sup_{\substack{\|\lambda\|_{L^2}=1 \\ 
    \lambda\geq 0\\
    |v(x)|=1
    }}
    \left\|P_{st}\left(\frac{\lambda}{\sqrt{6}}
    \left(I_3-3v \otimes v\right)\right)
    \right\|_{L^2}^2=r^2<1.
\end{equation}
Then for all $S\in\mathcal{L}^2_{st}$,
\begin{equation}
    \|\lambda_3-\lambda_2^+\|_{L^2}\geq 
    \frac{1-r}{\sqrt{2}} \|S\|_{L^2},
\end{equation}
where $\lambda_1(x)\leq \lambda_2(s)\leq \lambda_3(x)$
are the eigenvalues of $S(x)$ 
and $\lambda_2^+=\max(0,\lambda_2)$.
\end{theorem}

\begin{proof}
Fix $S\in\mathcal{L}^2_{st}$. Then we know that 
\begin{equation}
    S=
    \lambda_1 v_1\otimes v_1
    + \lambda_2 v_2\otimes v_2
    +\lambda_3 v_3 \otimes v_3,
\end{equation}
where $v_1,v_2,v_3$ are the normalized eigenvectors associated with the eigenvalues $\lambda_1\leq\lambda_2\leq\lambda_3$ of $S$.
We will define $\Tilde{S}$ by
\begin{equation}
    \tilde{S}(x)=
    \begin{cases}
    (\lambda_3(x)-\lambda_2(x))
    (v_3(x)\otimes v_3(x)-
    v_1(x)\otimes v_1(x)),
    & \lambda_2(x)>0 \\
    S(x), 
    &\lambda_2(x)\leq 0
    \end{cases}.
\end{equation}
Now we will observe that
\begin{equation}
    S=\lambda_2^+\left(I_3-3v_1\otimes v_1\right)
    +\tilde{S}.
\end{equation}
This is obvious when $\lambda_2\leq 0.$
When $\lambda_2>0,$ we can see that
\begin{align}
    \lambda_2^+\left(I_3-3v_1\otimes v_1\right)
    +\tilde{S}
    &=
    \lambda_2\left(I_3-3 v_1\otimes v_1\right)
    +(\lambda_3-\lambda_2)
    (v_3\otimes v_3-v_1\otimes v_1) \\
    &=
    (-\lambda_2-\lambda_3)v_1\otimes v_1
    +\lambda_2 v_2\otimes v_2
    +\lambda_3 v_3 \otimes v_3 \\
    &=
    \lambda_1 v_1\otimes v_1
    +\lambda_2 v_2\otimes v_2
    +\lambda_3 v_3 \otimes v_3 \\
    &=
    S,
\end{align}
using the fact that 
$\tr(S)=\lambda_1+\lambda_2+\lambda_3=0$
and $I_3=v_1\otimes v_1+v_2\otimes v_2+v_3\otimes v_3$.

Next we will show that 
\begin{equation}
    \left\|\tilde{S}\right\|_{L^2}\leq \sqrt{2}
    \left\|\lambda_3-\lambda_2^+\right\|_{L^2}.
\end{equation}
When $\lambda_2(x)>0,$ then
$\tilde{S}(x)=(\lambda_3(x)-\lambda_2(x))
(v_3(x)\otimes v_3(x)-v_1(x)\otimes v_1(x))$,
and so we can conclude that
\begin{equation}
    \left|\tilde{S}(x)\right|^2=
    2\left(\lambda_3(x)-\lambda_2^+(x)
    \right)^2.
\end{equation}
When $\lambda_2(x)\leq 0,$ then $\tilde{S}(x)=S(x)$
and $\lambda_2^+(x)=0$.
We know that $\lambda_1 \leq \lambda_2 \leq 0,$ 
and so consequently $\lambda_1\lambda_2\geq 0$.
Therefore we may conclude that 
\begin{align}
    \lambda_3^2
    &=
    (-\lambda_1-\lambda_2)^2 \\
    &=
    \lambda_1^2+\lambda_2^2+2\lambda_1\lambda_2 \\
    &\geq \lambda_1^2+ \lambda_2^2.
\end{align}
Therefore we can compute that
\begin{align}
    \left|\tilde{S}(x)\right|^2
    &=
    \lambda_1^2+\lambda_2^2+\lambda_3^2 \\
    &\leq
    2\lambda_3^2 \\
    &=
    2\left(\lambda_3-\lambda_2^+\right)^2.
\end{align}
Therefore we can see that for all $x\in\mathbb{R}^3$,
\begin{equation}
    \left|\tilde{S}(x)\right|^2
    \leq 
    2 \left(\lambda_3(x)-\lambda_2^+(x)
    \right)^2.
\end{equation}
Integrating this inequality we can conclude that
\begin{equation} \label{TildeBound}
    \left\|\tilde{S}\right\|_{L^2}\leq \sqrt{2}
    \left\|\lambda_3-\lambda_2^+\right\|_{L^2}.
\end{equation}

Using the fact that $S\in\mathcal{L}^2_{st}$,
and that therefore $S=P_{st}(S),$ we can see that
\begin{equation}
    S=P_{st}\left(\lambda_2^+
    \left(I_3-3v_1\otimes v_1\right)\right)
    +P_{st}\left(\tilde{S}\right).
\end{equation}
Applying the triangle inequality, the hypothesis on the sup of the projection of max-mid matrices \eqref{SupHypothesis}, 
and the bound \eqref{TildeBound},
we may conclude that
\begin{align}
    \|S\|_{L^2}
    &\leq
    \left\| P_{st}\left(\lambda_2^+\left(I_3
    -3v_1\otimes v_1\right)\right)\right\|_{L^2}
    +\left\|P_{st}\left(\tilde{S}\right)\right\|_{L^2} \\
    &\leq
    \sqrt{6}r\left\|\lambda_2^+\right\|_{L^2}
    +\left\|\tilde{S}\right\|_{L^2} \\
    & \leq
    r\|S\|_{L^2}+\sqrt{2}
    \left\|\lambda_3-\lambda_2^+\right\|_{L^2},
\end{align}
where we have used the fact that 
$0\leq \lambda_2^+(x)\leq\frac{1}{\sqrt{6}}|S(x)|$,
with equality if and only if $\lambda_2(x)=\lambda_3(x)$.
Rearranging terms we can see that
\begin{equation}
    \left\|\lambda_3-\lambda_2^+\right\|_{L^2}
    \geq \frac{1-r}{\sqrt{2}} \|S\|_{L^2},
\end{equation}
and this completes the proof.
\end{proof}

\begin{corollary}
Suppose
\begin{equation}
    \sup_{\substack{\|\lambda\|_{L^2}=1 \\ 
    \lambda\geq 0\\
    |v(x)|=1
    }}
    \left\|P_{st}\left(\frac{\lambda}{\sqrt{6}}
    \left(I_3-3v \otimes v\right)\right)
    \right\|_{L^2}^2=r^2<1.
\end{equation}
Then for all $S\in\mathcal{L}^2_{st}$,
\begin{equation}
    \|\lambda_3-\lambda_2\|_{L^2}\geq 
    \frac{1-r}{\sqrt{2}} \|S\|_{L^2}.
\end{equation}
\end{corollary}

\begin{proof}
Observing that 
$\lambda_3 \geq \lambda_2^+ \geq \lambda_2$,
we can conclude that
\begin{equation}
    0\leq \lambda_3-\lambda_2^+ \leq \lambda_3-\lambda_2.
\end{equation}
Therefore we can apply Theorem \ref{EigenGap}
and conclude that
\begin{align}
    \left\|\lambda_3-\lambda_2\right\|_{L^2} 
    &\geq
    \left\|\lambda_3-\lambda_2^+\right\|_{L^2} \\
    &\geq 
    \frac{1-r}{\sqrt{2}} \|S\|_{L^2}.
\end{align}
\end{proof}

Finally, as discussed in the introduction, we will reduce the question of the projection of maxmid matrices onto the strain constraint space to the question of the projection of rank one matrices onto the strain constraint space. We will now prove the identity \eqref{RankOneReduction}, which is restated for the reader's convenience.

\begin{proposition}
This supremum of the projection of $\mathcal{L}^2_{maxmid}$ matrices onto the space $\mathcal{L}^2_{st}$ can also be expressed in terms of the projection of rank one matrices onto the space $\mathcal{L}^2_{st}$ as follows:
\begin{equation}
    \sup_{\substack{\|\lambda\|_{L^2}=1 \\ 
    \lambda\geq 0\\
    |v(x)|=1
    }}
    \left\|P_{st}\left(\frac{\lambda}{\sqrt{6}}
    \left(I_3-3v \otimes v\right)\right)
    \right\|_{L^2}^2
    =
    \frac{3}{2}\sup_{\|w\|_{L^4}=1}
    \left\|P_{st}(w \otimes w)\right\|_{L^2}^2.
\end{equation}
\end{proposition}

\begin{proof}
For all $\|\lambda\|_{L^2}=1, \lambda\geq 0$, and for all $v\in L^\infty$ such that $|v(x)|=1$ almost everywhere,
define $w\in L^4$ by
\begin{equation}
    w(x) =\sqrt{\lambda(x)} v(x).
\end{equation}
Note that 
\begin{align}
    \|w\|_{L^4}^4
    &=
    \int_{\mathbb{R}^3}\lambda(x)^2 |v(x)|^4 \diff x \\
    &=
    \|\lambda\|_{L^2}^2 \\
    &=
    1.
\end{align}
Likewise we can see that
\begin{align}
    P_{st}\left(\frac{\lambda}{\sqrt{6}}
    \left(I_3-3v \otimes v\right)\right)
    &=
    -\frac{3}{\sqrt{6}} P_{st}\left(
    \sqrt{\lambda}v \otimes \sqrt{\lambda}v\right) \\
    &=
    -\frac{3}{\sqrt{6}} P_{st}\left(
    w \otimes w \right),
\end{align}
and therefore
\begin{align}
     \sup_{\substack{\|\lambda\|_{L^2}=1 \\ 
    \lambda\geq 0\\
    |v(x)|=1
    }}
    \left\|P_{st}\left(\frac{\lambda}{\sqrt{6}}
    \left(I_3-3v \otimes v\right)\right)
    \right\|_{L^2}^2
    &=
    \sup_{\|w\|_{L^4}=1}
    \left\|-\frac{3}{\sqrt{6}}P_{st}(w \otimes w)\right\|_{L^2}^2 \\
    &=
    \frac{3}{2}\sup_{\|w\|_{L^4}=1}
    \left\|P_{st}(w \otimes w)\right\|_{L^2}^2.
\end{align}
This completes the proof.
\end{proof}

\section{A viscous Hamilton-Jacobi type structure for the Navier--Stokes matrix potential evolution equation}
\label{MatrixPotentialSection}

In this section, we will consider the evolution equation for the matrix potential of the velocity in the context of Navier--Stokes equations. To begin, we consider an identity involving the divergence of a matrix and the Helmholtz projections onto the space of strain matrices and divergence free vector fields.

\begin{proposition} \label{DivergenceProjection}
For all $M\in \mathcal{L}^2$
\begin{equation}
    \divr P_{st}(M)= P_{df}(\divr(M))
\end{equation}
\end{proposition}

\begin{proof}
Applying the orthogonal decomposition in Theorem \ref{HelmholtzOne} let
\begin{equation}
    M=2\nabla_{sym}u+\Hess(-\Delta)^{-1}f+Q,
\end{equation}
where $u\in\mathbf{\dot{H}}^1_{df}, f\in L^2, 
Q\in \mathcal{L}^2_{divfree}$.
Taking the divergence of $M$, we find that
\begin{equation}
    \divr(M)=\Delta u-\nabla f.
\end{equation}
We know that $\nabla \cdot \Delta u=0,$ so applying the Helmholtz projection we find that
\begin{equation}
    P_{df}(\divr(M))=\Delta u.
\end{equation}
Likewise, it is clear by definition that
\begin{equation}
    P_{st}(M)=2\nabla_{sym}u,
\end{equation}
and so
\begin{equation}
    \divr P_{st}(M)=\Delta u.
\end{equation}
This completes the proof.
\end{proof}

\begin{remark}
Consider matrix
    \begin{equation}
        M=2(-\Delta)^{-1}\nabla_{sym}u.
    \end{equation}
Assuming $u$ has sufficient regularity and $\nabla\cdot u=0$, then we can see that $M$ is in the strain space, and by Theorem \ref{CharacterizationARMA},
\begin{equation}
    u=-\divr(M).
\end{equation}
Now consider the evolution equation
\begin{equation} \label{MatrixEvolutionRemark}
    \partial_t M-\Delta M
    -P_{st}(u\otimes u)
    =0,
\end{equation}
where $u=-\divr(M)$.
Applying Proposition \ref{DivergenceProjection}, we can see that
\begin{equation}
    \divr P_{st}(u\otimes u)
    =
    P_{df} \divr (u\otimes u),
\end{equation}
and so we can see that if we take the divergence of both sides of equation \eqref{MatrixEvolutionRemark} (and multiply by negative one), we obtain the Navier--Stokes equation,
\begin{equation}
    \partial_t u-\Delta u 
    +P_{df}\divr (u\otimes u)
    =0.
\end{equation}

This implies that the evolution equation \eqref{MatrixEvolutionRemark} describes the evolution of the matrix potential in the Navier--Stokes problem, something we will show more rigourously below in the context of mild solutions.
This is significant, as the viscous Hamilton-Jacobi equation,
\begin{equation}
    \partial_t-\Delta f-|\nabla f|^2=0,
\end{equation}
has global smooth solutions. 
In particular, $e^f$ satisfies the heat equation, and so we have
\begin{equation}
    f(\cdot,t)=\log \left(e^{t\Delta} 
    e^{f^0}\right).
\end{equation}
This approach cannot, unfortunately be used to obtain regularity for the Navier--Stokes equation, as the projection $P_{st}$ is a singular integral operator and the matrix structure makes the derivatives of the exponential considerably more complicated.
It does however, suggest another structure which could be used in order to search for new controlled quantities for the Navier--Stokes equation.
\end{remark}

We now define mild solutions for the Navier--Stokes matrix potential equation, and prove an equivalence with mild solutions of the Navier--Stokes equation.

\begin{definition} \label{MildMatrix}
$M\in C\left(\left[0,T_{max}\right);
\dot{\mathcal{H}}^1_{st}\cap\dot{\mathcal{H}}^2_{st}\right)$ 
is a mild solution of the Navier--Stokes matrix potential equation if for all $0<t<T_{max}$,
\begin{equation}
    M(\cdot,t)=e^{t\Delta}M^0
    +\int_0^t e^{(t-\tau)\Delta}
    P_{st}(\divr(M)\otimes \divr(M))(\cdot,\tau) \diff\tau.
\end{equation}
\end{definition}

\begin{definition} \label{MildNS}
$u\in C\left(\left[0,T_{max}\right);
\mathbf{H}^1_{df}\right)$ 
is a mild solution of the Navier--Stokes equation 
if for all $0<t<T_{max}$,
\begin{equation}
    u(\cdot,t)=e^{t\Delta}u^0
    -\int_0^t e^{(t-\tau)\Delta}
    P_{df} \divr (u \otimes u)(\cdot,\tau) \diff\tau.
\end{equation}
\end{definition}

\begin{proposition} \label{MildEquivalence}
There is an equivalence for mild solutions of the Navier--Stokes equation and the Navier--Stokes matrix potential equation:
$M\in C\left(\left[0,T_{max}\right);
\mathcal{\dot{H}}^1_{st}\cap\dot{\mathcal{H}}^2_{st}\right)$ 
is a mild solution of the Navier--Stokes matrix potential equation if and only if
$u\in C\left(\left[0,T_{max}\right);\mathbf{H}^1_{df}\right)$ is a mild solution of the Navier--Stokes equation,
where $u=-\divr(M)$ and $M=2\nabla_{sym}(-\Delta)^{-1}u$
\end{proposition}

\begin{proof}
Suppose $M\in C\left(\left[0,T_{max}\right);
\dot{\mathcal{H}}^1_{st}\cap\dot{\mathcal{H}}^2_{st}\right)$ is a mild solution of the Navier--Stokes matrix potential equation,
and let $u=-\divr(M)$, noting that this implies that
$u\in C\left(\left[0,T_{max}\right)
;\mathbf{H}^1_{df}\right)$.
By definition we have for all $0<t<T_{max}$
\begin{align}
    M(\cdot,t)
    &=
    e^{t\Delta}M^0
    +\int_0^t e^{(t-\tau)\Delta}
    P_{st}(\divr(M)\otimes \divr(M))(\cdot,\tau) \diff\tau \\
    &=
    e^{t\Delta}M^0
    +\int_0^t e^{(t-\tau)\Delta}
    P_{st}(u\otimes u)(\cdot,\tau) \diff\tau.
\end{align}
Taking the negative divergence of both sides of this equation, and putting derivatives on $M^0$ and $P_{st}(u\otimes u)$ in the convolution rather than on the heat kernel, 
we find that for all $0<t<T_{max}$
\begin{align}
    u(\cdot,t)
    &=
    e^{t\Delta} u^0
    -\int_0^t e^{(t-\tau)\Delta} 
    \divr P_{st}(u\otimes u)(\cdot,\tau) \diff\tau \\
    &=
    e^{t\Delta} u^0
    -\int_0^t e^{(t-\tau)\Delta} 
    P_{df}\divr (u\otimes u)(\cdot,\tau) \diff\tau \\
    &=
    e^{t\Delta}u^0
    -\int_0^t e^{(t-\tau)\Delta}
    P_{df}((u\cdot\nabla) u)(\cdot,\tau) \diff\tau,
\end{align}
using Proposition \ref{DivergenceProjection}, and the fact that $\nabla \cdot u=0$.
Note that the formal differentiation is entirely valid because of the smoothing due to the heat kernel.
This means that we can conclude that 
$u\in C\left(\left[0,T_{max}\right);
\mathbf{H}^1_{df}\right)$
is a mild solution of the Navier--Stokes equation, and this completes one direction of the proof.

Now suppose 
$u\in C\left(\left[0,T_{max}\right);\mathbf{H}^1_{df}\right)$
is a mild solution of the Navier--Stokes equation, 
and let $M=2\nabla_{sym}(-\Delta)^{-1}u$,
noting that this implies
$M\in C\left(\left[0,T_{max}\right);
\dot{\mathcal{H}}^1_{st}\cap\dot{\mathcal{H}}^2_{st}\right)$.
By definition we have for all $0<t<T_{max}$
\begin{equation}
    u(\cdot,t)=
    e^{t\Delta} u^0
    -\int_0^t e^{(t-\tau)\Delta} 
    P_{df}\divr (u\otimes u)(\cdot,\tau) \diff\tau.
\end{equation}
Taking 2$\nabla_{sym}(-\Delta)^{-1}$ of both sides of this equation and applying Proposition \ref{DivergenceProjection},
we find that for all $0<t<T_{max}$
\begin{align}
    M(\cdot,t)
    &=
    e^{t\Delta}M^0-\int_0^t e^{(t-\tau)\Delta}
    2\nabla_{sym}(-\Delta)^{-1}P_{df}\divr(u\otimes u)
    (\cdot,\tau) \diff\tau \\
    &=
    e^{t\Delta}M^0-\int_0^t e^{(t-\tau)\Delta}
    2\nabla_{sym}(-\Delta)^{-1}\divr P_{st}(u\otimes u)
    (\cdot,\tau) \diff\tau.
\end{align}
Observing that by definition 
$P_{st}(u\otimes u)\in\mathcal{L}^2_{st},$
and applying Theorem \ref{CharacterizationARMA},
we can see that
\begin{equation}
    P_{st}(u\otimes u)=
    -2\nabla_{sym}\divr(-\Delta)^{-1} P_{st}(u\otimes u).
\end{equation}
Recalling that $u=-\divr(M)$, we find that
for all $0<t<T_{max}$
\begin{align}
    M(\cdot,t)
    &=
    e^{t\Delta}M^0+\int_0^t e^{(t-\tau)\Delta}
    P_{st}(u\otimes u)
    (\cdot,\tau) \diff\tau \\
    &=
    e^{t\Delta}M^0+\int_0^t e^{(t-\tau)\Delta}
    P_{st}(\divr(M)\otimes \divr(M))
    (\cdot,\tau) \diff\tau.
\end{align}
Therefore we can conclude that
$M\in C\left(\left[0,T_{max}\right);
\dot{\mathcal{H}}^1_{st}\cap\dot{\mathcal{H}}^2_{st}\right)$
is a mild solution to the Navier--Stokes matrix potential equation and this completes the proof.
\end{proof}

Using Proposition \ref{MildEquivalence}, we will be able to deal with the theory of mild solutions of the Navier--Stokes matrix potential equation by reducing it to the theory of mild solutions for the Navier--Stokes equation, but we will first need to define the space of bounded mean oscillation functions.

For a function $f\in L^1_{loc}\left(\mathbb{R}^d\right)$,
define the average value over a cube $Q\in\mathbb{R}^d$
to be
\begin{equation}
    f_Q=\frac{1}{|Q|}\int_{Q}f(y)\diff y,
\end{equation}
and define the mean oscillation over the cube to be
\begin{equation}
    \frac{1}{|Q|}\int |f(y)-f_Q| \diff y.
\end{equation}
The space of bounded mean oscillation functions is defined to be the space where the mean oscillation is uniformly bounded.
The standard definition of the $BMO$ semi-norm is the supremum of the mean oscillation over all cubes,
\begin{equation}
    \|f\|_{BMO}=\sup_{Q\subset \mathbb{R}^d}
    \frac{1}{|Q|}\int |f(y)-f_Q| \diff y,
\end{equation}
however we will instead use the Carleson measure characterization as in \cite{KochTataru}, which is equivalent and will be more convenient for our purposes.

\begin{definition}
For all $f\in L^1_{loc}
\left(\mathbb{R}^d\right)$,
\begin{equation}
    \|f\|_{BMO}
    =
    \left(\sup_{\substack{x\in\mathbb{R}^3
    \\ R>0}}
    \frac{1}{|B(x,R)|}
    \int_{B(x,R)} \int_0^{R^2}
    \left|e^{t\Delta}\nabla f\right|^2
    \diff t \diff y\right)^\frac{1}{2};
\end{equation}
and
\begin{equation}
    BMO=\left\{f\in L^1_{loc}\left(\mathbb{R}^d\right):
    \|f\|_{BMO}<+\infty \right\}.
\end{equation}

We define the space of $BMO$ functions in the strain constraint space using the distributional definition, Definition \ref{DistributionDef}.
In particular, if $M\in BMO\left(\mathbb{R}^3;
\mathcal{S}^{3\times 3}\right)$,
then $M\in BMO_{st}$ if and only if
\begin{equation}
    \tr(M)=0,
\end{equation}
and
\begin{equation} \label{BMOstCondition}
    M+
    2\nabla_{sym}\divr(-\Delta)^{-1}M
    =0.
\end{equation}
Finally we define the space of derivatives of $BMO_{st}$ functions as
\begin{equation}
    BMO^{-1}_{df}=\left\{-\divr(M): 
    M\in BMO_{st}\right\},
\end{equation}
and norm this space by
\begin{equation}
    \|-\divr(M)\|_{BMO^{-1}}=\|M\|_{BMO}.
\end{equation}
\end{definition}

\begin{remark}
Note that the condition \eqref{BMOstCondition} is the same condition proven by the first author for $\mathcal{L}^2_{st}$ in \cite{MillerStrain}.
If we take the Riesz transform to be given by $R=\nabla(-\Delta)^{-\frac{1}{2}}$,
then this condition can be restated in terms of the Riesz transform as
\begin{equation}
    M+2R_{sym} R\cdot M=0.
\end{equation}
The Riesz transform is a bounded linear map from $BMO$ to $BMO$, so condition \eqref{BMOstCondition} is clearly well defined.

Further note that if $\divr(M)=0$, then by the condition \eqref{BMOstCondition}, we can see that $M=0$. This implies that $\divr: BMO_{st}\to BMO^{-1}_{df}$ is a bijection, and so the $M\in BMO_{st}$ satisfying $u=-\divr(M)$ is unique, and the $BMO^{-1}$ norm is also well defined.
\end{remark}

\begin{remark}
Note that $\|\cdot\|_{BMO}$ is not a norm for $BMO\left(\mathbb{R}^3;
\mathcal{S}^{3\times 3}\right)$,
because $\|\cdot\|_{BMO}$ applied to any constant function yields zero; however, $\|\cdot\|_{BMO}$ is a norm modulo the constant functions.
Furthermore, the Riesz-type transformation constraint on strain matrices given by \eqref{BMOstCondition} removes nonzero constant matrices from strain space $BMO_{st}$, because the divergence of a constant matrix is zero.
In particular, for all $M\in BMO_{st}$, if $\|M\|_{BMO}=0$, then $M$ is constant, which implies that $\divr(M)=0$, and therefore
\begin{equation}
    M=-2\nabla_{sym}
    (-\Delta)^{-1}\divr(M)
    =0.
\end{equation}
We can then conclude that $\|\cdot\|_{BMO}$ is a norm on the space $BMO_{st}$, which also immediately implies that $\|\cdot\|_{BMO^{-1}}$ is a norm on $BMO^{-1}_{df}$ when combined with the uniqueness of the matrix potential stated above.
\end{remark}

\begin{theorem} \label{MildExistenceNS}
For all initial data $u^0\in\dot{\mathbf{H}}^1_{df}$ 
there exists a unique mild solution of the Navier--Stokes equation
$u\in C\left(\left[0,T_{max}\right);\mathbf{H}^1_{df}\right)$ 
locally in time for some $T_{max}>0$.
In particular, there exists a universal constant $C$ independent of $u^0$ such that 
\begin{equation}
    T_{max}\geq \frac{C}
    {\left\|\nabla u^0\right\|_{L^2}^4}.
\end{equation}
This solution satisfies the energy equality,
with for all $0<t<T_{max}$
\begin{equation}
    \frac{1}{2}\|u(\cdot,t)\|_{L^2}^2
    +\int_0^t\|\nabla u(\cdot,\tau)\|_{L^2}^2 \diff\tau
    = \frac{1}{2}\left\|u^0\right\|_{L^2}^2.
\end{equation}
Furthermore, there exists an $\epsilon>0$, such that if 
$\left\|u^0\right\|_{BMO^{-1}}<\epsilon$,
then $T_{max}=+\infty$.
\end{theorem}

\begin{remark}
Fujita and Kato proved \cite{KatoFujita} the local existence of mild solutions for $\dot{H}^1$ initial data, with a uniform-in-norm lower bound on lifespan.
Leray first proved the energy equality for strong solutions---and energy inequality for weak solutions---in his ground-breaking work \cite{Leray}.
Finally, Koch and Tataru proved \cite{KochTataru} the global existence of strong solutions with small initial data in $BMO^{-1}$.
\end{remark}

\begin{corollary} \label{MildExistenceMatrix}
For all initial data $M^0\in
\dot{\mathcal{H}}^1_{st}\cap\dot{\mathcal{H}}^2_{st}$
there exists a unique mild solution of the Navier--Stokes matrix potential equation
$M\in C\left(\left[0,T_{max}\right);
\dot{\mathcal{H}}^1_{st}\cap\dot{\mathcal{H}}^2_{st}\right)$
locally in time for some $T_{max}>0$.
In particular, there exists a universal constant $C$ independent of $M^0$ such that 
\begin{equation} \label{ExistenceBound}
    T_{max}\geq \frac{C}{\|-\Delta M^0\|_{L^2}^4}.
\end{equation}
This solution satisfies the energy equality,
with for all $0<t<T_{max}$
\begin{equation} \label{MatrixEnergyEquality}
    \|\divr(M)(\cdot,t)\|_{L^2}^2
    +\int_0^t\|-\Delta M(\cdot,\tau)\|_{L^2}^2 \diff\tau
    = \left\|\divr\left(M^0\right)\right\|_{L^2}^2.
\end{equation}
Furthermore, there exists an $\epsilon>0$, such that if 
$\left\|M^0\right\|_{BMO}<\epsilon$,
then $T_{max}=+\infty$.
\end{corollary}

\begin{proof}
Fix $M^0 \in \mathcal{\dot{H}}^1_{st}
\cap\mathcal{\dot{H}}^2_{st}$, 
and let $u^0=-\divr(M^0)$, 
noting that $u^0\in\mathbf{H}^1_{df}$.
Applying Theorem \ref{MildExistenceNS}, we find that there exists a unique mild solution to the Navier--Stokes equation
$u\in C\left(\left[0,T_{max}\right);
\mathbf{\dot{H}^1}_{df}\right)$, where
\begin{equation}
    T_{max}<\frac{C}
    {\left\|\nabla u^0\right\|_{L^2}^4},
\end{equation}
and for all $0<t<T_{max}$,
\begin{equation}
    \frac{1}{2}\|u(\cdot,t)\|_{L^2}^2
    +\int_0^t\|\nabla u(\cdot,\tau)\|_{L^2}^2 \diff\tau
    = \frac{1}{2}\left\|u^0\right\|_{L^2}^2.
\end{equation}
For all $0\leq t<T_{max},$ let 
$M(\cdot,t)=2\nabla_{sym}(-\Delta)^{-1}u(\cdot,t)$,
noting that 
$M\in C\left(\left[0,T_{max}\right);
\mathcal{\dot{H}}^1_{st}\cap \mathcal{\dot{H}}^2_{st}\right)$.
Applying Proposition \ref{MildEquivalence}, it is clear that $M$ is a mild solution of the Navier--Stokes matrix potential equation.
We can also see that this mild solution satisfies the initial value problem with $M(\cdot,0)=M^0$ by
recalling that $u^0=-\divr\left(M^0\right)$,
and applying Theorem \ref{CharacterizationARMA},
\begin{align}
    M(\cdot,0)
    &=
    2\nabla_{sym}(-\Delta)^{-1}u^0\\
    &=
    -2\nabla_{sym}(-\Delta)^{-1}\divr\left(M^0\right) \\
    &=
    M^0.
\end{align}

We have now established the local-in-time existence and uniqueness of mild solutions to the Navier--Stokes matrix potential equation using its equivalence to the Navier--Stokes equation.
Observing that $u=-\divr(M)$, we can see that
\begin{equation}
    \|\divr(M)\|_{L^2}^2 =\|u\|_{L^2}^2.
\end{equation}
Observing that $-\Delta M=2\nabla_{sym}u$,
we can apply Proposition \ref{isometry} to find that
\begin{align}
    \|-\Delta M\|_{L^2}^2 
    &= 
    4\|\nabla_{sym} u\|_{L^2}^2 \\
    &=
    2\|\nabla u\|_{L^2}^2
\end{align}
Using these two inequalities, \eqref{ExistenceBound} and \eqref{MatrixEnergyEquality} follow from their analogues in Theorem \ref{MildExistenceNS}.

Finally, we observe that if $\left\|M^0\right\|_{BMO}<\epsilon$, then
\begin{align}
    \left\|u^0\right\|_{BMO^{-1}}
    &=
    \left\|-\divr\left(M^0\right)\right\|_{BMO^{-1}} \\
    &=
    \left\|M^0\right\|_{BMO} \\
    &< \epsilon,
\end{align}
and therefore $T_{max}=+\infty$.
This completes the proof.
\end{proof}

\begin{remark}
Koch and Tataru use a different norm for the space $BMO^{-1}_{df}$ in \cite{KochTataru}, taking the norm to be
    \begin{equation}
        \|f\|_{\widetilde{BMO}^{-1}}
        =
        \left(\sup_{\substack{x\in\mathbb{R}^3
        \\ R>0}}
        \frac{1}{|B(x,R)|}
        \int_{B(x,R)} \int_0^{R^2}
        \left|e^{t\Delta}f\right|^2
        \diff t \diff y\right)^\frac{1}{2};
    \end{equation}
however, these norms are equivalent. It is immediately apparent that for all $f\in BMO$,
\begin{equation}
    \|f\|_{BMO}
    =
    \|\nabla f\|_{\widetilde{BMO}^{-1}},
\end{equation}
which motivates the definition of $\widetilde{BMO}^{-1}$.
The equivalence of these norms is shown in one direction in \cite{KochTataru}.
In particular, they show in the proof of Theorem 1 that for all 
$M\in BMO(\mathbb{R}^3;
\mathcal{S}^{3\times 3})$,
\begin{align}
    \|-\divr(M)\|_{\widetilde{BMO}^{-1}}
    &\leq  
    \|\nabla M
    \|_{\widetilde{BMO}^{-1}} \\
    &=
    \|M\|_{BMO}
\end{align}
If we let $u=-\divr(M)$, then we find that
\begin{equation}
    \|u\|_{\widetilde{BMO}^{-1}}
    \leq \|u\|_{BMO^{-1}}.
\end{equation}

The reverse direction is not proven explicitly in \cite{KochTataru}, because Koch and Tataru do not consider the strain constraint space $BMO_{st}$. It is obviously not true that there is a constant $C>0$ such that for all $M\in BMO\left(\mathbb{R}^3;
\mathcal{S}^{3\times 3}\right)$,
\begin{equation} \label{BoundBMO}
   \|M\|_{BMO}\leq C\|-\divr(M)\|_{\widetilde{BMO}^{-1}}, 
\end{equation}
because we can simply take any nonzero, non-constant matrix, such that $\divr(M)=0$,
and then \eqref{BoundBMO} will fail.
However, when we restrict to the space $BMO_{st}$, such an inequality will hold. Using the boundedness of Riesz transform from $BMO \to BMO$ and from $BMO^{-1} \to BMO^{-1}$ and the constraint condition
$M=-2\nabla_{sym}\divr(-\Delta)^{-1}M$, we find that
\begin{align}
    \|M\|_{BMO}
    &=
    2\left\|\nabla_{sym}\divr
    (-\Delta)^{-1}M
    \right\|_{BMO} \\
    &=
    2\left\|R_{sym}(-\Delta)^{-\frac{1}{2}}
    \divr(M)\right\|_{BMO} \\
    &\leq
    C\left\|(-\Delta)^{-\frac{1}{2}}
    \divr(M)\right\|_{BMO} \\
    &=
    C\left\|\nabla
    (-\Delta)^{-\frac{1}{2}}
    \divr(M)
    \right\|_{\widetilde{BMO}^{-1}} \\
    &\leq
    C\|-\divr(M)
    \|_{\widetilde{BMO}^{-1}}.
\end{align}
See in particular the proof of Lemma 4.1 in \cite{KochTataru}.
Again recalling that $u=-\divr(M)$, we can see that
\begin{equation}
    \|u\|_{BMO^{-1}}
    \leq
    C \|u\|_{\widetilde{BMO}^{-1}},
\end{equation}
and the norms are equivalent.
\end{remark}

\section{The two dimensional case and the general
\textit{d}-dimensional case}
\label{d-DimensionalSection}

In this section we will consider the general case for the Helmholtz decomposition for the the space of symmetric matrices over $\mathbb{R}^d, d\geq 2$. Because the structure of the proof is broadly similar to the case $d=3$, we will sketch the proofs without getting exhaustively to the linear algebra details.
We will begin with the case $d=2$.

\begin{theorem} \label{Helmholtz2D}
For $\mathcal{L}^2\left(\mathbb{R}^2\right)$ we have the following decomposition
\begin{equation}
\mathcal{L}^2\left(\mathbb{R}^2\right)
=\mathcal{L}^2_{st}\left(\mathbb{R}^2\right)
\oplus \mathcal{L}^2_{Hess}\left(\mathbb{R}^2\right)
\oplus L^2_{\widetilde{Id}}\left(\mathbb{R}^2\right).
\end{equation}
\end{theorem}

\begin{proof}
We begin by letting, for all $\xi\neq 0$,
\begin{equation}
    \mathcal{B}=\left\{
    \frac{\xi}{|\xi|},\mu\right\}
\end{equation}
be an orthogonal basis for $\mathbb{R}^2$.
We can then observe that
\begin{equation}
    \Tilde{\mathcal{B}}
    =
    \left\{
    \frac{\xi}{|\xi|}\otimes \frac{\xi}{|\xi|},
    \frac{1}{\sqrt{2}}\left(
    \frac{\xi}{|\xi|}\otimes\mu
    +\mu\otimes \frac{\xi}{|\xi|}\right),
    \mu\otimes \mu
    \right\}
\end{equation}
is an orthogonal basis for $\mathcal{S}^{2\times 2}$, the space of complex valued symmetric matrices.
Finally we observe that $M\in \mathcal{L}^2_{Hess}$ 
if and only if for all $\xi\in\mathbb{R}^2$
\begin{equation}
    \hat{M}(\xi) \in \spn\left\{
    \frac{\xi}{|\xi|}\otimes \frac{\xi}{|\xi|}\right\};
\end{equation}
$M\in \mathcal{L}^2_{st}$ if and only if 
\begin{equation}
    \hat{M}(\xi) \in \spn\left\{
    \frac{1}{\sqrt{2}}\left(
    \frac{\xi}{|\xi|}\otimes\mu
    +\mu\otimes \frac{\xi}{|\xi|}\right)\right\};
\end{equation}
$M\in \mathcal{L}^2_{\widetilde{Id}}$ 
if and only if for all $\xi\in\mathbb{R}^2$
\begin{equation}
    \hat{M}(\xi) \in \spn\left\{
    \mu\otimes\mu\right\}.
\end{equation}
The orthogonal decomposition then follows immediately by taking the orthogonal decomposition point-wise in Fourier space. 
\end{proof}

\begin{proposition} \label{DivFree2D}
    In the case $d=2$, we have
    \begin{equation}
        \mathcal{L}^2_{\widetilde{Id}}\left(\mathbb{R}^2\right)
        =
        \mathcal{L}^2_{divfree}\left(\mathbb{R}^2\right)
    \end{equation}
\end{proposition}

\begin{proof}
Suppose $M \in \mathcal{L}^2_{divfree}$.
We know from the decomposition in Theorem \ref{Helmholtz2D} that we can express the Fourier transform by
\begin{equation}
    \hat{M}(\xi)
    = 
    f(\xi) \frac{\xi}{|\xi|} \otimes \frac{\xi}{|\xi|}
    +\frac{g(\xi)}{\sqrt{2}}\left(
    \frac{\xi}{|\xi|}\otimes \mu
    +\mu\otimes \frac{\xi}{|\xi|}\right)
    +h(\xi) \mu\otimes \mu. 
\end{equation}
Because $M\in\ker(\divr)$, we find that
for all $\xi\in\mathbb{R}^2$
\begin{align}
    \hat{M}(\xi)\xi
    &=
    f(\xi)|\xi|\xi+\frac{g}{\sqrt{2}}|\xi|\mu \\
    &=0,
\end{align}
and so we can conclude that $f,g=0$ almost everywhere.
Therefore we can conclude that for all $\xi\in\mathbb{R}^2$,
\begin{equation}
    \hat{M}(\xi) = h(\xi) \mu\otimes\mu,
\end{equation}
and consequently that $M\in\mathcal{L}^2_{\widetilde{Id}}$.

Likewise, if $M\in\mathcal{L}^2_{\widetilde{Id}}$,
then for all $\xi\in\mathbb{R}^2$
\begin{equation}
    \hat{M}(\xi)=h(\xi)\mu\otimes \mu.
\end{equation}
We can then see clearly that 
for all $\xi\in\mathbb{R}^2$,
\begin{equation}
    \hat{M}(\xi)\xi=0,
\end{equation}
and so $M\in\ker(\divr)$.
This completes the proof.
\end{proof}

\begin{proposition} \label{trdivTrivial}
The space of trace and divergence free, symmetric
$2 \times 2$ is trivial, with
\begin{equation}
    \mathcal{L}^2_{tr\&divfree}\left(\mathbb{R}^2\right)=
    \ker(\tr)\cap\ker(\divr)
    =\left\{\mathbf{0}\right\}.
\end{equation}
\end{proposition}.

\begin{proof}
Suppose $M \in \mathcal{L}^2_{tr\&divfree}$.
Because $M\in \ker(\divr)$, we can apply Proposition \ref{DivFree2D} to see that $M\in\mathcal{L}^2_{\widetilde{Id}}$,
and therefore 
\begin{equation}
    \hat{M}(\xi)=h(\xi) \mu\otimes \mu.
\end{equation}
The trace operator commutes with the Fourier transform, so we can also conclude that
\begin{equation}
    \tr\left(\hat{M}(\xi)\right)=h(\xi)=0,
\end{equation}
almost everywhere $\xi\in\mathbb{R}^2$.
Therefore, $M=0$, and this completes the proof.
\end{proof}

\begin{theorem} \label{HelmholtzGeneral}
For all $d\geq 3$, we have the orthogonal decomposition
\begin{align}
    \mathcal{L}^2 \left(\mathbb{R}^d\right)
    &= \label{FirstLine}
    \mathcal{L}^2_{st}\left(\mathbb{R}^d\right)
    \oplus \mathcal{L}^2_{Hess}\left(\mathbb{R}^d\right)
    \oplus \mathcal{L}^2_{divfree}
    \left(\mathbb{R}^d\right) \\
    &= \label{SecondLine}
    \mathcal{L}^2_{st}\left(\mathbb{R}^d\right)
    \oplus \mathcal{L}^2_{Hess}\left(\mathbb{R}^d\right)
    \oplus L^2_{\widetilde{Id}}\left(\mathbb{R}^d\right)
    \oplus \mathcal{L}^2_{tr\&divfree}
    \left(\mathbb{R}^d\right)
\end{align}

\begin{proof}
For each $\xi\in\mathbb{R}^d, \xi\neq 0$,
let
\begin{equation}
    \mathcal{B}=
    \left\{
    \frac{\xi}{|\xi|},\mu_1,...,\mu_{d-1}
    \right\}
\end{equation}
be an orthonormal basis for $\mathbb{R}^d$.
Let
\begin{align}
    E_1
    &=
    \spn\left\{
    \frac{\xi}{|\xi|}\otimes \frac{\xi}{|\xi|}
    \right\} \\
    E_2
    &=
    \spn\left\{
    \frac{1}{\sqrt{2}}\left(
    \frac{\xi}{|\xi|}\otimes\mu_k
    +\mu_k\otimes \frac{\xi}{|\xi|}\right)
    \right\}_{1\leq k\leq d-1} \\
    E_3
    &=
    \spn\left\{
    \frac{1}{\sqrt{2}}\left(
    \mu_j\otimes\mu_k
    +\mu_k\otimes \mu_j\right)
    \right\}_{1\leq j< k\leq d-1} \\
    E_4
    &=
    \spn\left\{
    \frac{1}{\sqrt{2}}\left(\mu_k\otimes \mu_k
    -\mu_{k+1}\otimes \mu_{k+1}\right)
    \right\}_{1\leq k\leq d-2} \\
    E_5
    &=
    \spn\left\{
    \frac{1}{\sqrt{d-1}}\left(
    \sum_{k=1}^{d-1}\mu_k\otimes\mu_k
    \right)\right\}.
\end{align}
Observe that the set vectors above form an orthonormal basis for $S^{d\times d}$, the space of complex valued symmetric matrices. In particular we have the orthogonal decomposition
\begin{equation}
    S^{d\times d}=
    E_1 \oplus E_2 \oplus E_3 \oplus E_4 \oplus E_5.
\end{equation}
As is the case for the three dimensional decomposition, the proof follows from a straightforward dimension counting argument (orthogonality is trivial), but we will leave the linear algebra details to the reader.

Also analogous to the three dimensional case,
we have for 
$M\in \mathcal{L}^2\left(\mathbb{R}^d\right)$:
\begin{itemize}
    \item 
$M\in \mathcal{L}^2_{Hess}$ if and only if
for all $\xi\in\mathbb{R}^d, \hat{M}(\xi)\in E_1$
    \item
$M\in \mathcal{L}^2_{st}$ if and only if
for all $\xi\in\mathbb{R}^d, \hat{M}(\xi)\in E_2$;
    \item
$M\in \mathcal{L}^2_{divfree}$ if and only if
for all $\xi\in\mathbb{R}^d, 
\hat{M}(\xi)\in E_3\oplus E_4 \oplus E_5$
    \item
$M\in \mathcal{L}^2_{tr\&divfree}$ if and only if
for all $\xi\in\mathbb{R}^d, \hat{M}\in E_3 \oplus E_4$
    \item
$M\in \mathcal{L}^2_{\widetilde{Id}}$ if and only if
for all $\xi\in\mathbb{R}^d, 
\hat{M}(\xi)\in E_5$.
\end{itemize}
Using this characterization, we can simply perform the orthogonal decomposition point-wise in Fourier space, and the result follows.
\end{proof}

\begin{remark}
Note that Proposition \ref{DivFree2D} and Theorem \ref{Helmholtz2D} imply that \eqref{FirstLine} in Theorem \ref{HelmholtzGeneral} holds also for the case $d=2$. Furthermore, Proposition \ref{trdivTrivial} and Theorem \ref{Helmholtz2D} imply that while \eqref{SecondLine} also holds for the case of $d=2$, the last space in the decomposition is trivial. This means the decomposition holds in general for $d\geq 2$.
\end{remark}
\end{theorem}

\section{The anti-symmetric case}

In this section, we will consider the analogous decomposition for anti-symmetric matrix valued functions. We will prove Theorem \ref{HelmholtzAntiSymIntro}, which is broken into two pieces for the reader's covenience.

\begin{theorem} \label{HelmholtzAntiSym}
For all $d\geq 2$, the space of anti-symmetric matrix valued functions 
$L^2\left(\mathbb{R}^d;\mathcal{A}^{d\times d}\right)$ has the following orthogonal decomposition:
\begin{equation}
    L^2\left(\mathbb{R}^d;\mathcal{A}^{d\times d}\right)
    =
    \mathbb{L}^2_{vort} \oplus \mathbb{L}^2_{divfree}. 
\end{equation}
\end{theorem}

\begin{proof}
As in the proof of Theorem \ref{HelmholtzGeneral},
for each $\xi\in\mathbb{R}^d, \xi\neq 0$,
let
\begin{equation}
    \mathcal{B}=
    \left\{
    \frac{\xi}{|\xi|},\mu_1,...,\mu_{d-1}
    \right\}
\end{equation}
be an orthonormal basis for $\mathbb{R}^d$.
Let
\begin{align}
    E_1 &=
    \spn\left\{\frac{1}{\sqrt{2}}\left(
    \frac{\xi}{|\xi|}\otimes \mu_k-
    \mu_k\otimes \frac{\xi}{|\xi|}\right)
    \right\}_{1\leq k\leq d-1} \\
    E_2 &=
    \spn\left\{\frac{1}{\sqrt{2}}\left(
    \mu_j\otimes \mu_k-\mu_k\otimes\mu_j\right)
    \right\}_{1\leq j< k\leq d-1}.
\end{align}
It is again elementary linear algebra to check that
for all $\xi\in\mathbb{R}^d, \xi\neq 0$,
\begin{equation}
    A^{d\times d}=E_1 \oplus E_2.
\end{equation}
It also can be shown working from the Fourier space side of the definitions that for each $M\in \mathbb{L}^2(\mathbb{R}^d)$,
\begin{itemize}
    \item 
    $M\in \mathbb{L}^2_{vort}$ if and only if
    for all $\xi\in\mathbb{R}^d, 
    \hat{M}(\xi) \in E_1$
    \item 
    $M\in \mathbb{L}^2_{divfree}$ if and only if
    for all $\xi\in\mathbb{R}^d, 
    \hat{M}(\xi) \in E_2$.
\end{itemize}
The result then follows by taking the decomposition point-wise in Fourier space.
\end{proof}

\begin{corollary}
For all $d\geq 2$, the space of matrix valued functions
$L^2\left(\mathbb{R}^d;\mathbb{R}^{d\times d}\right)$
has the following orthogonal decomposition:
\begin{align}
    L^2\left(\mathbb{R}^d;\mathbb{R}^{d\times d}\right)
    &=
    L^2\left(\mathbb{R}^d;\mathcal{S}^{d\times d}\right)
    \oplus
    L^2\left(\mathbb{R}^d;\mathcal{A}^{d\times d}\right) \\
    &=
    \mathcal{L}^2_{st} \oplus \mathcal{L}^2_{Hess} 
\oplus \mathcal{L}^2_{divfree} \oplus
    \mathbb{L}^2_{vort} \oplus \mathbb{L}^2_{divfree}.
\end{align}
\end{corollary}

\begin{proof}
We can see immediately that 
\begin{equation}
    L^2\left(\mathbb{R}^d;\mathbb{R}^{d\times d}\right)
    =
    L^2\left(\mathbb{R}^d;\mathcal{S}^{d\times d}\right)
    \oplus
    L^2\left(\mathbb{R}^d;\mathcal{A}^{d\times d}\right),
\end{equation}
simply by taken the symmetric and anti-symmetric part of the matrix at each point and noting that they are orthogonal. The result then follows from Theorems \ref{HelmholtzGeneral} and \ref{HelmholtzAntiSym}
\end{proof}

\begin{remark}
    In the special case, $d=3$, we can represent anti-symmetric matrices with vectors as follows.
    We will take
    \begin{equation}
        A=\left(
        \begin{array}{ccc}
            0 & w_3 & -w_2  \\
            -w_3 & 0 & w_1 \\
            w_2 & -w_1 & 0
        \end{array}
        \right).
    \end{equation}
    In this case, the matrix multiplication can be described by a vector cross product, 
    with for all $v\in\mathbb{R}^3$,
    \begin{equation}
        Av= v \times w.
    \end{equation}
    
    We can also see that in this case,
    \begin{equation}
        \divr A= -\nabla \times w.
    \end{equation}
    Using the fact that gradients are curl free, and applying the Helmholtz decomposition for anti-symmetric matrices and for vector fields, we can see that,
    \begin{align}
        \mathbb{L}^2_{vort}
        &=\left\{
        \left(\begin{array}{ccc}
            0 & w_3 & -w_2  \\
            -w_3 & 0 & w_1 \\
            w_2 & -w_1 & 0
        \end{array}\right):
        w\in \mathbf{L}^2_{df}\right\} \\
        \mathbb{L}^2_{divfree}
        &=\left\{
        \left(\begin{array}{ccc}
            0 & w_3 & -w_2  \\
            -w_3 & 0 & w_1 \\
            w_2 & -w_1 & 0
        \end{array}\right):
        w\in \mathbf{L}^2_{gr} \right\}
    \end{align}
    This decomposition is why the vorticity can be represented as a vector in the case when $d=3$, with 
    \begin{equation}
        \nabla_{asym} u=
        \frac{1}{2}
        \left(\begin{array}{ccc}
            0 & \omega_3 & -\omega_2  \\
            -\omega_3 & 0 & \omega_1 \\
            \omega_2 & -\omega_1 & 0
        \end{array}\right),
    \end{equation}
    where $\omega=\nabla\times u$,
    for all $u\in\dot{\mathbf{H}}^1$.

It is worth noting that the fact that all gradients are in the kernel of the curl for $d=3$, 
generalizes for all $d \geq 2$, to the statement that
    for all $f\in \dot{H}^2$,
    \begin{equation}
        \nabla_{asym}\nabla f=0,
    \end{equation}
    which follows immediately from the fact that all Hessians are symmetric.
\end{remark}

\section*{Acknowledgements}
This publication was supported in part by the Fields Institute for Research in the Mathematical Sciences while the first author was in residence during the Fall 2020 semester. Its contents are solely the responsibility of the authors and do not necessarily represent the official views of the Fields Institute.
This material is based upon work supported by the National Science Foundation under Grant No. DMS-1440140 while the first author participated in a program that was hosted by the Mathematical Sciences Research Institute in Berkeley, California, during the Spring 2021 semester.
The second author's research is supported by NSERC Discovery Grant RGPIN-2020-06829.

\bibliographystyle{plain}
\bibliography{Bib}
\end{document}